\documentclass[11pt,a4paper,final,oneside,openright,reqno]{amsart}

 \usepackage{latexsym, amsfonts, amsthm, amsmath, amssymb, tikz, enumitem, booktabs, float,color,multicol,epsf, url,epsfig,epstopdf, array, setspace, graphicx, inputenc, csquotes, xcolor,tikz-cd}

\usepackage[margin=1in,footskip=0.25in]{geometry}
\usepackage{titlesec}
\usepackage{amsfonts}
\usepackage{amsmath}
\usepackage{amssymb}
\usepackage{url}
\usepackage{hyperref}
\usepackage{color}
\usepackage{lipsum}
\usepackage{mathtools}
 \usepackage[T1]{fontenc}
\usepackage{tikz-cd}
\tikzcdset{scale cd/.style={every label/.append style={scale=#1},
    cells={nodes={scale=#1}}}}
 \usepackage{url}
 \usepackage{hyperref}
\usepackage{pdfpages}
\usepackage{graphicx}
\usepackage{titletoc}
\usepackage{marginnote}
\usepackage{enumitem}
\usepackage[all]{hypcap}

\titlecontents{section}
  [0em] 
  {\vspace{0.5em}} 
  {\thecontentslabel\quad} 
  {} 
  {\hfill\contentspage} 
  [\vspace{0.5em}] 

\titlecontents{subsection}
  [0em] 
  {\vspace{0.3em}} 
  {\thecontentslabel\quad} 
  {} 
  {\hfill\contentspage} 
  [\vspace{0.3em}] 

\let\oldcup\cup
\let\oldcap\cap
\usepackage{mathabx}
\let\cup\oldcup
\let\cap\oldcap

\titleformat{\section}
{\normalfont\Large\bfseries}{\thesection}{1em}{}
\titleformat{\subsection}
{\normalfont\large\bfseries}{\thesubsection}{1em}{}

\numberwithin{equation}{section}

\newcommand{\mylabel}[2]{#2\def\@currentlabel{#2}\label{#1}}

\def\HK{H_K}
\def\HdK{H^*_K}
\def\psiK{\Psi_K}

\def\Honenull{H^1_0(\mathbb{T},\mathbb{R}^{2n})}

\def\pio{\pi_0}

\def\conv{\operatorname{Conv}(\mathbb{R}^{2n})}

\def\chara{\operatorname{Char}}
\def\sys{\operatorname{Sys}}
\def\ev{\text{ev}_0}
\newcommand{\crit}{\operatorname{crit}}

\newcommand{\indsys}{\operatorname{ind}_{\operatorname{sys}}^{S^1}}

\def\indfr{\operatorname{ind}_{\operatorname{FR}}}

 \theoremstyle{plain}
\newtheorem*{theorem*}{Theorem}
\newtheorem{Theorem}{Theorem}[section]
\newtheorem{Corollary}[Theorem]{Corollary}

\theoremstyle{definition}

\newtheorem{Definition}[Theorem]{Definition}

\theoremstyle{plain}
\newtheorem{theorem}{Theorem}[subsection]
\newtheorem{corollary}[theorem]{Corollary}
\newtheorem{lemma}[theorem]{Lemma}
\newtheorem{proposition}[theorem]{Proposition}
\theoremstyle{definition}
\newtheorem{remark}[theorem]{Remark}
\newtheorem*{remark*}{Remark}
\newtheorem{example}[theorem]{Example}
\newtheorem{question}{Question}

\title{Systolic $S^1$-index and characterization of non-smooth Zoll convex bodies}

\author{Stefan Matijevi{\'c}}

\begin{document}

\begin{abstract}
We define the systolic $S^1$-index of a convex body as the Fadell–Rabinowitz index of the space of generalized systoles associated with its boundary. We show that this index is a symplectic invariant. Using the systolic $S^1$-index, we introduce the notion of generalized Zoll convex bodies and prove that this definition coincides with the classical one when the convex body satisfies the uniqueness of systoles property, that is, when through every point passes at most one systole. Moreover, we show that generalized Zoll convex bodies can be characterized in terms of their Gutt–Hutchings capacities, and we prove that the space of generalized Zoll convex bodies is closed in the space of all convex bodies. As a corollary, we establish that if the interior of a convex body is symplectomorphic to the interior of a ball, then the convex body is generalized Zoll, and in particular Zoll if it satisfies the uniqueness of systoles property. Finally, we discuss several examples.

\end{abstract}

\maketitle

\tableofcontents

\section{Introduction}

We say that $K \subset \mathbb{R}^{2n}$ is a convex body if $K$ is the closure of a bounded open convex set. The boundary of such a set is a hypersurface,which is a Lipschitz submanifold of $\mathbb{R}^{2n}$ and, as such, has a tangent plane and an outer normal vector that are defined almost everywhere. Moreover, we can define the tangent cone and the normal cone at every point of $\partial K$ (see \cite{Cla83}).

We define the outward normal cone of $\partial K$ at $x \in \partial K$ as 
\[
N_{\partial K}(x) = \{\eta \in \mathbb{R}^{2n}\backslash\{0\} \mid \langle \eta, x - y \rangle \geq 0, \quad \forall y \in K\}.
\]

If $\partial K$ is differentiable at $x$, then $N_{\partial K}(x) = \mathbb{R}_{+} n$, where $n$ is the unit outer normal vector at $x$. Let $J_0$ denote the standard complex structure of $\mathbb{R}^{2n}$, and let $\mathbb{T} := \mathbb{R}/\mathbb{Z}$.

Following \cite{Cla81}, we define the set of generalized closed characteristics on $\partial K$, denoted by $\chara(\partial K)$, as the set
\[
\chara(\partial K) := \{\gamma: \mathbb{T} \to \partial K \mid \dot{\gamma}(t) \in J_0 N_{\partial K}(\gamma(t)), \ \text{a.e.}\}/\sim,
\]

where $\gamma: \mathbb{T} \to \partial K$ is an absolutely continuous loop satisfying the condition that there exist positive constants $m_\gamma, M_\gamma > 0$ such that  
\[
m_\gamma \leq |\dot{\gamma}(t)| \leq M_\gamma, \quad \text{a.e.}
\]

The equivalence relation $\sim$ is defined as follows. We say that $\gamma_1$ is equivalent to $\gamma_2$ (i.e., $\gamma_1 \sim \gamma_2$) if there exists a bi-Lipschitz homeomorphism $\tau: \mathbb{T} \to \mathbb{T}$ such that $\tau(0) = 0$ and $
\gamma_1(t) = \gamma_2(\tau(t)).$
This homeomorphism must preserve orientation. Let $\gamma: \mathbb{T} \to \mathbb{R}^{2n}$ be an absolutely continuous loop. We define the action of $\gamma$ as 
\[
\mathcal{A}(\gamma) := \int\limits_{\gamma} \lambda = \frac{1}{2} \int\limits_\mathbb{T} \langle \dot{\gamma}(t), J_0 \gamma(t) \rangle \, dt,
\]
where $\lambda$ is any primitive 1-form of the standard symplectic form
\[
\omega_0 = \sum_{j=1}^n dx_j \wedge dy_j.
\]

From Stokes' formula, it follows that the action is invariant under reparametrizations that preserve orientation. Thus, the action is well-defined on the space $\chara(\partial K)$, is finite, and can easily be shown to be positive for every element of $\chara(\partial K)$. We define the spectrum of $\partial K$ as
\[
\sigma(\partial K) := \{ \mathcal{A}(\gamma) \mid \gamma \in \chara(\partial K) \}.
\]
 The subset of $\chara(\partial K)$ consisting of generalized closed characteristics with minimal action is denoted by $\sys(\partial K)$, and it is always non-empty (see \cite{Cla81, AO14}). Elements of this set are called generalized systoles. To introduce an $S^1$-action on the spaces $\chara(\partial K)$ and $\sys(\partial K)$, we need a choice of parametrization. 

Assuming the origin is contained in the interior of $K$, we can define a positively $2$-homogeneous function $\HK: \mathbb{R}^{2n} \to \mathbb{R}$ such that $\HK^{-1}(1) = \partial K$. Since $\HK$ is convex, the subdifferential of $\HK$ at $x$ is defined for every $x \in \mathbb{R}^{2n}$ as 
\[
\partial \HK(x) = \{\eta \in \mathbb{R}^{2n} \mid \HK(y) - \HK(x) \geq \langle \eta, y - x \rangle, \quad \forall y \in \mathbb{R}^{2n}\}.
\]

For every $x \in \mathbb{R}^{2n}$, the set $\partial \HK(x)$ is a non-empty convex compact set, which is precisely $\{\nabla \HK(x)\}$ when $\HK$ is differentiable at $x$. Moreover, the multifunction $\partial \HK$ is upper semi-continuous. See \cite{Cla83} for the proof of these properties.

The function $\HK$ allows us to define generalized closed characteristics on $\partial K$ as follows:
\[
\chara(\partial K) = \left\{\gamma: \mathbb{T} \to \partial K \text{ absolutely continuous} \mid \dot{\gamma}(t) \in T J_0 \partial \HK(\gamma(t)) \ \text{a.e.}, \ \text{for some } T > 0 \right\}.
\]

This definition is equivalent to the previous one in the set-theoretic sense. Indeed, any $\gamma \in \chara(\partial K)$, according to this definition, has a derivative that is bounded both from below and above. Moreover, $N_{\partial K}(x) = \mathbb{R}_+ \partial \HK(x)$ for all $x \in \partial K$. Hence, the conclusion follows.

In addition, generalized closed characteristics with a fixed action $T > 0$ are precisely those satisfying the equation
\[
\dot{\gamma}(t) \in T J_0 \partial \HK(\gamma(t)), \quad \text{a.e.}
\]

This follows from the fact that the $2$-homogeneity of $\HK$ implies that
\[
\HK(x) = \frac{1}{2} \langle \partial \HK(x), x \rangle, \quad x \in \mathbb{R}^{2n},
\]
where the equality holds in the set-theoretic sense (for any element of the corresponding subdifferential). Consequently, for $\gamma \in \chara(\partial K)$, we have
\[
\dot{\gamma}(t) \in T J_0 \partial \HK(\gamma(t)), \quad \text{a.e.},
\]
if and only if $\mathcal{A}(\gamma) = T$. In particular, we have
\[
\sys(\partial K) = \left\{ \gamma : \mathbb{T} \to \partial K \ \text{a. continuous} \mid \dot{\gamma}(t) \in T_{\min} J_0 \partial \HK(\gamma(t)) \ \text{a.e.}, \ T_{\min} = \min \sigma(\partial K) \right\}.
\]

Now, on the space $\chara(\partial K)$, there is a natural $S^1$-action defined by
\[
\theta \cdot \gamma = \gamma(\cdot - \theta), \quad \theta \in \mathbb{T}, \ \gamma \in \chara(\partial K).
\]

Since the $S^1$-action preserves the absolute continuity of a loop $\gamma: \mathbb{T} \to \mathbb{R}^{2n}$, the subsets of generalized closed characteristics with fixed action are $S^1$-invariant. In particular, $\sys(\partial K)$ is $S^1$-invariant.

Given that $\partial \HK$ is bounded on $\partial K$, it follows that the space $\sys(\partial K)$ is uniformly bounded in the $W^{1,\infty}$-norm. We endow $\sys(\partial K)$ with the compact-open topology (i.e., $C^0$-topology). As shown in Proposition \ref{uniformtopology}, this space is compact with respect to this topology, which is, in fact, equivalent to the weak$^*$ topology of $W^{1,\infty}$. In contrast, with respect to other topologies, such as the strong topology of $W^{1,1}$, this space is generally not compact (see Proposition \ref{w11norm}). Hence, the uniform norm provides a natural framework for studying this space when one aims to preserve compactness properties, as in the smooth setting.

\textbf{Fadell–Rabinowitz index}

In this paper, we employ Alexandrov–Spanier cohomology (see \cite{Spa66}), as it is the cohomology theory used in the definition of the Fadell–Rabinowitz index. 

Let $\mathbb{F}$ be a field, and let $X$ be a paracompact topological space equipped with an $S^1$-action. Following Borel’s construction, we define the $S^1$-equivariant cohomology of $X$ with coefficients in $\mathbb{F}$ as
\[
\operatorname{H}^*_{S^1}(X; \mathbb{F}) = \operatorname{H}^*(X \times_{S^1} ES^1; \mathbb{F}),
\]
where $ES^1 \to BS^1$ denotes the universal $S^1$-bundle over the classifying space, and
\[
X \times_{S^1} ES^1 := (X \times ES^1)/S^1.
\]

For instance, one may take $ES^1 := S^\infty \subset \mathbb{C}^\infty$ and $BS^1 := \mathbb{C}P^\infty$. The cohomology of the classifying space $BS^1$ is the ring $\operatorname{H}^*(BS^1; \mathbb{F}) \cong \mathbb{F}[e_B]$, where $e_B$ is a generator of $\operatorname{H}^2(BS^1; \mathbb{F})$. The projection map 
\[
\pi_2 : X \times_{S^1} ES^1 \to BS^1, \qquad \pi_2([x, y]) = [y],
\]
induces a homomorphism of cohomology rings
\[
\pi^*_2 : \operatorname{H}^*(BS^1; \mathbb{F}) \to \operatorname{H}^*_{S^1}(X; \mathbb{F}).
\]
We denote $e := \pi_2^* e_B$ as the fundamental class associated with the $S^1$-space $X$. The Fadell–Rabinowitz index of $X$ is then defined by
\[
\indfr(X; \mathbb{F}) = \sup_{i \geq 0} \{ i + 1 \mid e^i \neq 0 \},
\]
and we set $\indfr(\emptyset; \mathbb{F}) = 0$.

In what follows, we focus on the case $\mathbb{F} = \mathbb{Z}_2$, in accordance with the results of \cite{Mat24}. Nevertheless, analogous statements remain valid over any field, since \cite[Theorem~A]{Mat24} holds for arbitrary field coefficients.

\begin{Definition}\label{definitionofthesystolicS1index}
 Let $K \subset \mathbb{R}^{2n}$ be a convex body, and let $K'$ be any translation of $K$ whose interior contains the origin. We define the \textit{systolic $S^1$-index} of $K$ as 
\[
\indsys(K) = \indfr(\sys(\partial K')),
\] 
where the topology on $\sys(\partial K')$ is induced by the uniform norm.   
\end{Definition} As the notation suggests, this definition does not depend on the choice of the translation $K'$. This is a consequence of Theorem \ref{mainsystolics1index} below.

In \cite{GH18}, Gutt and Hutchings defined a monotone sequence of capacities $c_i^{GH}$ for the class of Liouville domains. As shown in \cite[Theorem A]{Mat24}, for convex bodies, the Gutt–Hutchings capacities coincide with the spectral invariants introduced by Ekeland and Hofer in \cite{EH87}.  

Using this result, we prove the following.

\begin{Theorem}\label{mainsystolics1index}
Let $K \subset \mathbb{R}^{2n}$ be a convex body whose interior contains the origin. Then, it holds that
\[
\indfr(\sys(\partial K)) = \max\{i \in \mathbb{N} \mid c_i^{GH}(K) = c_1^{GH}(K)\}.
\]

In particular, $\indsys(K)$ is well-defined for a convex body $K \subset \mathbb{R}^{2n}$, and it holds
\[
\indsys(K) = \max\{i \in \mathbb{N} \mid c_i^{GH}(K) = c_1^{GH}(K)\}.
\]
\end{Theorem}

Since $c_i^{GH}(K) \to \infty$ as $i \to \infty$, it follows that $\indsys(K)$ is finite. Moreover, from the previous theorem, we obtain the following immediate corollary.

\begin{Corollary}
The systolic $S^1$-index of convex bodies is a symplectic invariant. More precisely, if the interior of a convex body $K_1$ is symplectomorphic to the interior of a convex body $K_2$, then 
\[
\indsys(K_1) = \indsys(K_2).
\]
\end{Corollary}

\begin{remark*}
From Theorem \ref{mainsystolics1index} and \cite[Theorem 1.2]{GR24}, we deduce that
\[
\indsys(K) = \max\{i \in \mathbb{N} \mid c_i^{EH}(K) = c_1^{EH}(K)\},
\]
where $c_i^{EH}$ denotes the Ekeland–Hofer capacities introduced in \cite{EH89,EH90}.
\end{remark*}

Let $\conv$ denote the set of all convex bodies endowed with the Hausdorff-distance topology. Since $\indsys$ is finite for every convex body, the function 
\[
\indsys: \conv \to \mathbb{N}
\]
is well-defined. Moreover, the function is upper semicontinuous and uniformly bounded.

The upper semicontinuity of $\indsys$ is a straightforward corollary of the previous theorem (see claim (1) of Proposition \ref{uppersemicontuperbound}). The existence of a uniform bound follows from the existence of the Loewner–Behrend–John ellipsoid (see \cite{Joh48} or \cite[Appendix B]{Vit00}). Indeed, using this ellipsoid, we obtain explicit bounds (see claim (2) of Proposition \ref{uppersemicontuperbound}):
\begin{itemize}
    \item For every $K \in \conv$, we have $\indsys(K) \leq 4n^3$. If $K = -K$, then $\indsys(K) \leq 2n^2$.
\end{itemize}

In the case of more restrictive classes, we can obtain optimal estimates. Consider the standard $S^1$-action on $\mathbb{C}^n$ given by
\[
\theta \cdot z = e^{2\pi i \theta} z, \quad \theta \in \mathbb{T}, \ z \in \mathbb{C}^n.
\]
We say that $K \in \conv$ is $S^1$-invariant if, for every $\theta \in \mathbb{T}$, it holds that $\theta \cdot K = K.$

\begin{itemize}
    \item For every $K \in \conv$ such that $K$ is $S^1$-invariant, it holds that $\indsys(K) \leq n$.
    \item For every $K \in \conv$ whose boundary is of class $C^{1,1}$, it holds that $\indsys(K) \leq n$.
\end{itemize}

Since a ball in $\mathbb{R}^{2n}$ has systolic $S^1$-index equal to $n$, we conclude that these estimates are optimal. The estimate in the $S^1$-invariant case follows from claim (3) of Proposition \ref{uppersemicontuperbound}, while the estimate for convex bodies with $C^{1,1}$ boundary follows from Corollary \ref{systolics1indexnupperboundS1} (in this corollary, the estimate is actually established for a larger class of convex bodies described in Definition \ref{definitionuos}). We now use the systolic $S^1$-index to extend the definition of Zoll convex bodies to arbitrary convex bodies. First, we recall the definition of the Zoll property.

Let $K \subset \mathbb{R}^{2n}$ be a convex body with a $C^{1,1}$ boundary that is transverse to all lines through the origin (this condition is equivalent to requiring that the origin lies in $\text{int}(K)$). For such a convex body, $(K, \lambda_0)$ is a Liouville domain, where  
\[
\lambda_0 = \frac{1}{2} \sum_{j=1}^n (x_j \, dy_j - y_j \, dx_j)
\]
is the standard Liouville form, serving as a primitive of the standard symplectic form $\omega_0$.

The Reeb vector field $R$ on $\partial K$ is defined by the system of equations  
\[
i_R \lambda_0 = 1, \quad i_R \omega_0 = 0,
\]
which determines a globally defined flow. Notice that
\[
R(x) = J_0 \nabla \HK(x), \quad x \in \partial K.
\] 

A convex body $K$ with a $C^{1,1}$ boundary is called \textit{Zoll} if all of its Reeb orbits are closed and share a common minimal period. Equivalently, through every point of $\partial K$ there passes a systole. Note that the $C^{1,1}$ regularity of the boundary is required for the existence of the flow, but the definition based on the property that a systole passes through every point makes sense even without the existence of the flow. This notion can be extended to convex bodies satisfying the following property.

\begin{Definition}\label{definitionuos}
A convex body $K$ satisfies the \textit{uniqueness of systoles} property if, through every point, there passes at most one generalized systole, up to reparametrization.
  \end{Definition}

This class of convex bodies includes those with $C^{1,1}$ boundaries but is much broader, as it imposes no regularity conditions beyond those already ensured by convexity.  

\begin{Definition}\label{definitionzoll}
Let $K$ be a convex body that satisfies the uniqueness of systoles property. We say that $K$ is \textit{Zoll} if through every point a generalized systole passes. 
\end{Definition}

Beyond the class of convex bodies satisfying the uniqueness of systoles property, it is unclear how to dynamically extend the notion of Zoll convex bodies, since for such a convex body there exists at least one point through which multiple systoles pass. Motivated by the symplectic nature of the systolic $S^1$-index and its established properties, we propose the following definition.

\begin{Definition}
A convex body $K \subset \mathbb{R}^{2n}$ is a \textit{generalized Zoll convex body} if it satisfies  
\[
\indsys(K) \geq n.
\]
\end{Definition}

This definition is inspired by the work of Ginzburg, Gürel, and Mazzucchelli \cite{GGM21} and by \cite{Mat24}. From \cite[Corollary B.2]{Mat24}, it follows that a smooth and strongly convex body $K$ (i.e., one whose boundary $\partial K$ has positive sectional curvature everywhere) is Zoll if and only if $c^{GH}_1(K) = c^{GH}_n(K)$. The next theorem shows that the assumption of strong convexity can be replaced by mere convexity, and smoothness can be replaced by the uniqueness of systoles property, while the characterization in terms of Gutt--Hutchings capacities remains valid for generalized Zoll convex bodies.

\begin{Theorem}\label{maingeneralizedzoll}
The following statements hold:
\begin{enumerate}
    \item A convex body $K$ is generalized Zoll if and only if $c^{GH}_1(K) = c^{GH}_n(K)$.
    \item A convex body that satisfies the uniqueness of systoles property is generalized Zoll if and only if it is Zoll.
    \item The space of generalized Zoll convex bodies is closed in the space of all convex bodies with respect to the Hausdorff distance.
\end{enumerate}
\end{Theorem}

Statement (2) of this theorem justifies the term "generalized Zoll convex body." Note that, since the systolic $S^1$-index is invariant under symplectic maps, the definition of generalized Zoll convex bodies is independent of symplectomorphisms of their interiors. In other words, being a generalized Zoll convex body is a symplectic notion. Since a ball is a Zoll convex body, we have the following immediate corollary.

\begin{Corollary}\label{zollball}
If the interior of a convex body $K \subset \mathbb{R}^{2n}$ is symplectomorphic to the interior of a ball, then $K$ is a generalized Zoll convex body. In particular, if $K$ also satisfies the uniqueness of systoles property, then $K$ is Zoll.
\end{Corollary}

In addition, in Corollary~\ref{systolics1indexnupperboundS1}, we show that the only $S^1$-invariant generalized Zoll convex body is the ball. Hence, the only known examples of generalized Zoll convex bodies that are not Zoll in the usual sense (see Definition~\ref{definitionzoll}) arise from convex bodies that do not satisfy the uniqueness of systoles property and whose interiors are symplectomorphic to the interior of a ball (see \cite{Tra95, Sch05, LMS13, Rud22, ORS23}). One such example is the Lagrangian product of the unit ball in the $\|\cdot\|_\infty$-norm and the unit ball in the $\|\cdot\|_1$-norm, which we denote by $B_\infty \times B_1$.

As already proved, if $K \subset \mathbb{R}^{2n}$ satisfies the uniqueness of systoles property, then $K$ is a generalized Zoll convex body if and only if a systole passes through every point. However, this is not true in general. In particular, the condition that a systole passes through every point is not sufficient (see Example~\ref{polydicsevaluation} concerning the polydisc). We also show in Example~\ref{nonsmoothzollevaluation} that $B_\infty \times B_1$ is a generalized Zoll convex body that does not satisfy the uniqueness of systoles property, yet a generalized systole passes through every point.  

We conclude this introduction with a list of some open questions.

\begin{question}
What is the global maximum of the function $\indsys: \conv \to \mathbb{N}$? Does it hold that for every convex body $K \subset \mathbb{R}^{2n}$, $\indsys(K) \leq n$?
\end{question}

\begin{question}
If $K$ is a generalized Zoll convex body, does a generalized systole pass through every point? Is the converse true if $\sys(\partial K)$ is connected?
\end{question}

\begin{question}
If $K \subset \mathbb{R}^{2n}$ is a generalized Zoll convex body, is it true that the space $\sys(\partial K)$ has the cohomology type of a sphere $S^{2n-1}$?
\end{question}

\begin{question}
What is the form of Gutt–Hutchings capacities for $k > n$ for generalized Zoll convex bodies (see \cite{MR23} for the answer in the smooth strongly convex case)?
\end{question}

\begin{question}
Are there exotic generalized Zoll convex bodies, i.e., those with a systolic ratio different from that of a ball? In particular, are there generalized Zoll convex bodies that do not belong to the Hausdorff closure of the space of smooth Zoll convex bodies?  
\end{question}

\begin{question}
Are generalized Zoll convex bodies local maximizers of the systolic ratio in the space of convex bodies with respect to the Hausdorff distance topology, either in general or in the centrally symmetric case (see \cite{AB23, ABE23} for partial results in the smooth setting)?
\end{question}

\textbf{Acknowledgements.} I would like to express my gratitude to my advisor, Alberto Abbondandolo, for his invaluable guidance and support throughout my PhD studies. I would also like to thank Souheib Allout for several valuable discussions.  

This project is supported by the SFB/TRR 191 ‘Symplectic Structures in Geometry, Algebra and Dynamics,’ funded by the DFG (Projektnummer 281071066 - TRR 191).

\section{Spaces of generalized and centralized generalized systoles
}\label{sectiontopology}

To prove Theorem~\ref{mainsystolics1index}, we need to introduce the space of centered generalized systoles.

Let
\[
\pi_0(\gamma) = \gamma - \int\limits_{\mathbb{T}} \gamma(t) \, dt
\]
denote the orthogonal $L^2$ projection onto the space of closed curves in $\mathbb{R}^{2n}$ with zero mean.

We define the set of centralized generalized systoles, denoted by $\sys_0(\partial K)$, as
\[
\sys_0(\partial K) := \pi_0(\sys(\partial K)) = \{\pi_0(\gamma) \mid \gamma \in \sys(\partial K)\}.
\]

Alternatively, this space can be interpreted as the space of derivatives of generalized systoles.

\subsection{On the choice of topology}

In this subsection, we consider the problem of choosing a natural topology on the spaces $\sys(\partial K)$ and $\sys_0(\partial K)$. First, we recall some classical results from convex analysis.

Let $f: \mathbb{R}^m \to \mathbb{R}$ be a convex function. Such a function is locally Lipschitz continuous. Moreover, the directional derivative 
\[
f'(x;v) := \lim_{t \downarrow 0} \frac{f(x+tv)-f(x)}{t}
\]
is finite for every $x \in \mathbb{R}^m$ and $v \in \mathbb{R}^m$.

The function $f'(x; \cdot)$ is positively homogeneous, subadditive, and Lipschitz continuous. Additionally, the function $f'$ is upper semicontinuous as a function of $(x, v)$.

Let $\langle \cdot, \cdot \rangle$ denote the standard scalar product. We can express the subdifferential of $f$ at $x$, denoted by $\partial f(x)$, as 
\[
\partial f(x) := \{\xi \in \mathbb{R}^m \mid f'(x,v) \geq \langle \xi, v \rangle, \quad \forall v \in \mathbb{R}^m \}.
\]
The set $\partial f(x)$ is a non-empty, convex, and compact subset of $\mathbb{R}^m$. Moreover, $f'(x; \cdot)$ is the support function of $\partial f(x)$, and therefore it holds that
\[
f'(x;v) = \max\{\langle v, \xi \rangle \mid \xi \in \partial f(x)\}.
\]

For more details on these results, we refer the reader to \cite{Cla83}.

The next proposition is part of the general compactness theory for weak solutions of differential equations, as presented in \cite[Theorem 3.1.7]{Cla83}. For completeness, we provide a self-contained proof here.

\begin{proposition}\label{uniformtopology}
Let $K\subset \mathbb{R}^{2n}$ be a convex body whose interior contains the origin. On the spaces $\sys(\partial K)$ and $\sys_0(\partial K)$, the compact-open topology and weak$^*$-$W^{1,\infty}$ topology are equivalent. Moreover, the spaces $\sys(\partial K)$ and $\sys_0(\partial K)$ are compact in this topology.
\end{proposition}

\begin{proof}
Let $K$ be a convex body whose interior contains the origin. We can assume without loss of generality that the action on the space of generalized systoles is $1$. Then $\gamma \in \sys(\partial K)$ if and only if $\gamma: \mathbb{T} \to \partial K$ is Lipschitz and satisfies the weak equation
\begin{equation}\label{systoleequation}
    \dot{\gamma}(t) \in J_0 \partial \HK(\gamma(t)), \quad \text{a.e.}
\end{equation}

This equation is equivalent to the condition that 
\begin{equation}\label{systoleequationderivative}
\HK'(\gamma(t); v) + \langle J_0 \dot{\gamma}(t), v \rangle \geq 0, \quad \forall v \in \mathbb{R}^{2n}    
\end{equation}
holds for almost every $t \in \mathbb{T}$, due to the relation between the subdifferential and the directional derivative. From \eqref{systoleequation} and the fact that $\partial \HK$ is bounded on $\partial K$, we get that the space $\sys(\partial K)$ is uniformly bounded in the $W^{1,\infty}$-norm. Therefore, if $\gamma_i \in \sys(\partial K)$ is an arbitrary sequence, it has a convergent subsequence in the weak$^*$-$W^{1,\infty}$ topology. Weak$^*$-$W^{1,\infty}$ topology convergence implies uniform convergence (by Arzelà-Ascoli), and therefore we can assume that $\gamma_i$ converges to some $\gamma$ uniformly and in the weak$^*$-$W^{1,\infty}$ topology (we will keep the same notation for the subsequence). Since $\partial K$ is compact, we conclude that $\gamma: \mathbb{T} \to \partial K$. Moreover, the uniform $W^{1,\infty}$-bound on $\sys(\partial K)$ implies that $\gamma$ is Lipschitz. Now we need to show that $\gamma$ is a solution of equation \eqref{systoleequation}.

Let $U \subset \mathbb{T}$ be any measurable subset. From the reformulation of \eqref{systoleequation} given by \eqref{systoleequationderivative}, we have that 
\[
\int_{U} \left( \HK'(\gamma_i(t); v) + \langle J_0 \dot{\gamma}_i(t), v \rangle \right) \, dt \geq 0.
\]

By the upper semi-continuity of $\HK'$ and the uniform and weak$^*$-$W^{1,\infty}$ convergence of $\gamma_i$, the previous inequality implies the following:

\begin{align*}
 0 &\leq \limsup_{i \to \infty} \int_{U} \left( \HK'(\gamma_i(t); v) + \langle J_0 \dot{\gamma}_i(t), v \rangle \right) \, dt \\
 &\leq \limsup_{i \to \infty} \int_{U} \HK'(\gamma_i(t); v) \, dt + \limsup_{i \to \infty} \int_{U} \langle J_0 \dot{\gamma}_i(t), v \rangle \, dt \\
 &\leq \int_{U} \left( \HK'(\gamma(t); v) + \langle J_0 \dot{\gamma}(t), v \rangle \right) \, dt.
\end{align*}

This ensures the inequality 
\[
\HK'(\gamma(t); v) + \langle J_0 \dot{\gamma}(t), v \rangle \geq 0
\] 
almost everywhere. Continuity of $\HK'$ in the second variable ensures that the inequality 
\[
\HK'(\gamma(t); v) + \langle J_0 \dot{\gamma}(t), v \rangle \geq 0
\]
holds for every $v \in \mathbb{R}^{2n}$ outside a fixed measure-zero subset of $\mathbb{T}$. Therefore, we conclude that \eqref{systoleequation} holds, and hence $\gamma \in \sys(\partial K)$. Since the weak$^*$-$W^{1,\infty}$ topology is metrizable on bounded subsets of $W^{1,\infty}$, we have shown that $\sys(\partial K)$ is compact in this topology. This implies that any Hausdorff topology on $\sys(\partial K)$ which is coarser than the weak$^*$-$W^{1,\infty}$ topology is equivalent to it. In particular, this holds for the topology induced by the uniform norm. The same conclusions follow for $\sys_0(\partial K)$.

\end{proof}

Let $B_\infty \times B_1 \subset \mathbb{R}^4$ be the Lagrangian product of the unit ball in the $\|\cdot\|_\infty$-norm and the unit ball in the $\|\cdot\|_1$-norm.

\begin{proposition}\label{w11norm}
There exists a sequence $\gamma_n: \mathbb{T} \to \partial(B_\infty \times B_1)$ of piecewise smooth generalized systoles that converges uniformly to the piecewise smooth generalized systole $\gamma: \mathbb{T} \to \partial(B_\infty \times B_1)$ but does not converge to $\gamma$ in the $W^{1,1}$-norm.
\end{proposition}

\begin{proof}
Let $K = B_\infty \times B_1$. Since $\text{int}(K)$ is symplectomorphic to the ball  
\[
B(4) = \{(z_1, z_2) \in \mathbb{C}^2 \mid \pi(|z_1|^2 + |z_2|^2) < 4\},
\]
it follows that the generalized systoles on $\partial K$ have action 4 (see \cite{LMS13}). Indeed, by the symplectomorphism, we have $c^{GH}_1(K) = c^{GH}_1(\overline{B}(4)) = 4$. Moreover, since for convex bodies $c_1^{GH}$ represents the minimum of the action spectrum, i.e., the action on the space $\sys(\partial K)$, the conclusion follows.

The function $ \HK: \mathbb{R}^4 \to \mathbb{R} $ is defined as 
\[
\HK(x_1, x_2, y_1, y_2) = \max\{\|(x_1, x_2)\|_\infty^2, \|(y_1, y_2)\|_1^2\}.
\]

We define the loop $\gamma: \mathbb{T} \to \partial K$ as
\[
\gamma(t) = 
\begin{cases}
\big(-1 + 8t, 0, -1, 0\big), & t \in \big[0, \frac{1}{4}\big], \\
\big(1, 0, -1 + 8(t - \frac{1}{4}), 0\big), & t \in \big[\frac{1}{4}, \frac{1}{2}\big], \\
\big(1 - 8(t - \frac{1}{2}), 0, 1, 0\big), & t \in \big[\frac{1}{2}, \frac{3}{4}\big], \\
\big(-1, 0, 1 - 8(t - \frac{3}{4}), 0\big), & t \in \big[\frac{3}{4}, 1\big].
\end{cases}
\]

We claim that this is a generalized systole. This function is piecewise smooth and therefore Lipschitz. For $(x, 0, -1, 0)$ where $x \in [-1, 1]$, we have that $\partial \HK(x, 0, -1, 0)$ contains
\[
\text{conv}\{(0, 0, -2, -2), (0, 0, -2, 2)\},
\]
which is the convex hull of these two points. This follows from the upper semicontinuity of the subgradient as a multifunction and the fact that
\[
\partial \HK\left(x, 0, -1, 0\right) = \partial \|(0, -1)\|_1^2 = \text{conv}\{(-2, -2), (-2, 2)\}, \quad \forall x \in (-1, 1),
\]
which holds because $\HK(x_1, x_2, y_1, y_2) = \|(y_1, y_2)\|_1^2$ in the neighborhood of the point $(x, 0, -1, 0)$ when $x \in (-1, 1)$.

In particular, we have that $\partial \HK(x, 0, -1, 0)$ contains $(0, 0, -2, 0)$ as the convex combination of the two previously mentioned vectors. Therefore, $J_0 \partial \HK(x, 0, -1, 0)$ contains $(2, 0, 0, 0)$ for all $x \in [-1, 1]$. This implies that
\[
\dot{\gamma}(t) \in 4J_0 \partial \HK(\gamma(t)), \quad t \in \left[0, \frac{1}{4}\right].
\]

Using similar reasoning, we show that $J_0 \partial \HK(1, 0, y, 0)$ contains $(0, 0, 2, 0)$ for all $y \in [-1, 1]$.  Hence, we have that
\[
\dot{\gamma}(t) \in 4J_0 \partial \HK(\gamma(t)), \quad t \in \left[\frac{1}{4}, \frac{1}{2}\right].
\]

Analogously, we handle the cases when $t \in [\frac{1}{2}, \frac{3}{4}]$ and $t \in [\frac{3}{4}, 1]$. Combining all of this, we get that
\[
\dot{\gamma}^\pm(t) \in 4J_0 \partial \HK(\gamma(t)), \quad t \in \mathbb{T}.
\]

Hence, $\gamma$ is a closed characteristic, and since the minimal action on the space of characteristics is 4 (by the discussion at the beginning of the proof), we conclude that $\gamma \in \sys(\partial K)$.

Now, we will define a sequence of systoles $\gamma_n: \mathbb{T} \to \partial K$ that uniformly converges to $\gamma$.

\[
\gamma_n(t) =
\begin{cases}
\big(-1 + 8t,   8(t - \frac{2k}{8n}), -1, 0\big), & t \in \big[\frac{2k}{8n}, \frac{2k+1}{8n}\big], \quad k \in \{0, \dots, n-1\}, \\
\big(-1 + 8t,  \frac{1}{n} - 8(t - \frac{2k-1}{8n}), -1, 0\big), & t \in \big[\frac{2k-1}{8n}, \frac{2k}{8n}\big], \quad k \in \{1, \dots, n\}, \\
\big(1,0, -1 + 8(t - \frac{1}{4}), 0\big), & t \in \big[\frac{1}{4}, \frac{1}{2}\big], \\
\big(1 - 8(t - \frac{1}{2}),0, 1, 0\big), & t \in \big[\frac{1}{2}, \frac{3}{4}\big], \\
\big(-1, 0, 1 - 8(t - \frac{3}{4}), 0\big), & t \in \big[\frac{3}{4}, 1\big].
\end{cases}
\]

The approximating sequence is illustrated in Figure \ref{fig:aproximatingsequance}.

\begin{figure}[h]
  \centering
  \includegraphics[scale=3]{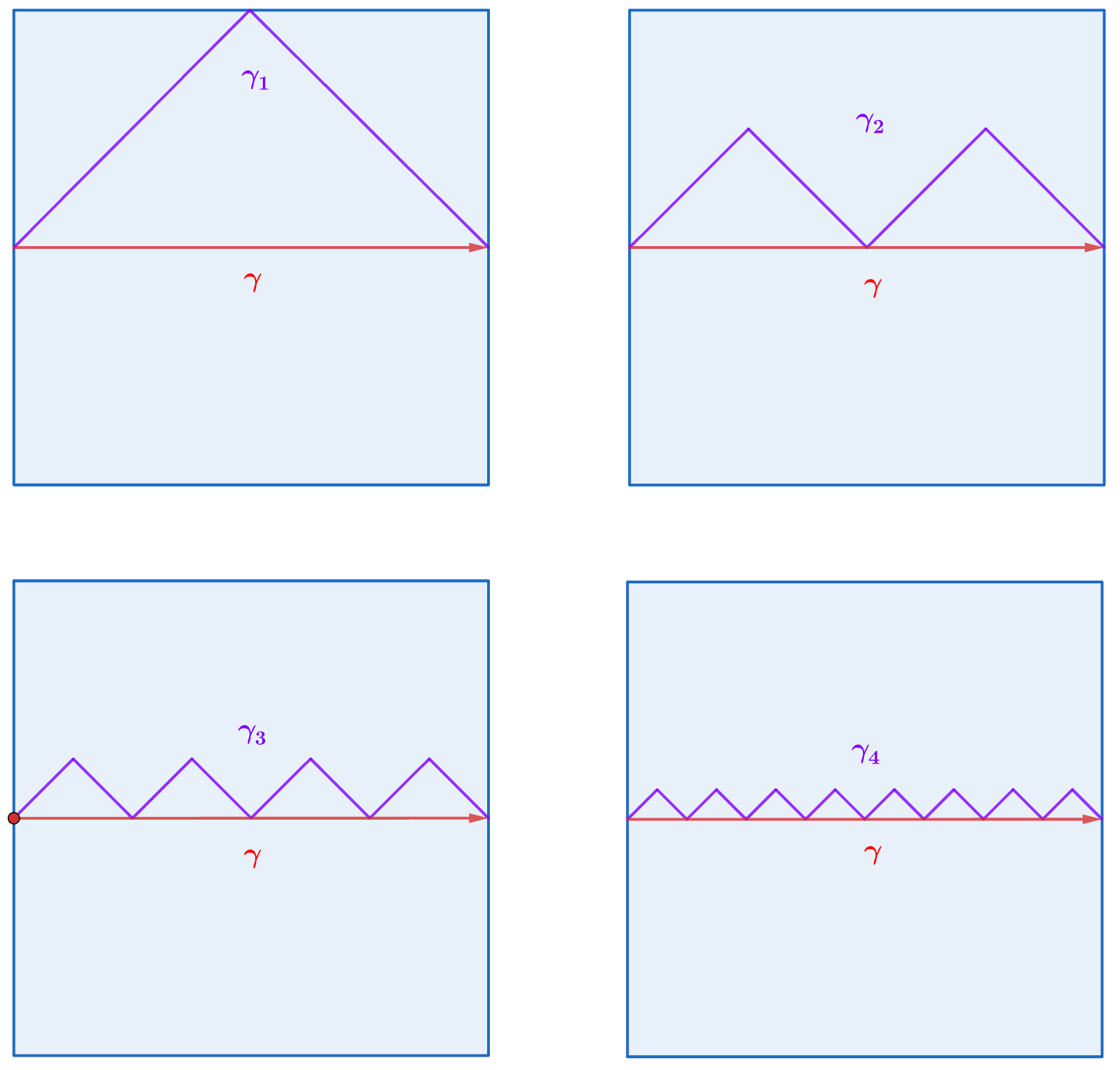}
  \caption{Approximating sequence $\gamma_n$ of $\gamma$.}
  \label{fig:aproximatingsequance}
\end{figure}

The image of $\gamma_n$ is indeed inside $\partial K$. We have that
\[
\partial \HK(x, y, -1, 0) = \text{conv}\{(0, 0, -2, -2), (0, 0, -2, 2)\},
\]
when $(x_1, x_2) \in \text{int}(B_\infty)$, which extends, due to the upper semicontinuity of $\partial \HK$, to boundary points of $B_\infty$ in the sense that the subdifferential contains $\text{conv}\{(0, 0, -2, -2), (0, 0, -2, 2)\}$. In particular, for $(x_1, x_2) \in B_\infty$, the set $J_0 \partial \HK(x_1, x_2, -1, 0)$ contains $(2, 2, 0, 0)$ and $(2, -2, 0, 0)$.  

Since it holds that
\[
\dot{\gamma}_n(t) = (8, 8, 0, 0), \quad t \in \left[\frac{2k}{8n}, \frac{2k+1}{8n}\right], \quad k \in \{0, \dots, n-1\},
\]
and
\[
\dot{\gamma}_n(t) = (8, -8, 0, 0), \quad t \in \left[\frac{2k-1}{8n}, \frac{2k}{8n}\right], \quad k \in \{1, \dots, n\},
\]
we conclude that
\[
\dot{\gamma}_n^\pm(t) \in 4J_0 \partial \HK(\gamma_n(t)), \quad t \in \left[0, \frac{1}{4}\right].
\]

Because $\gamma_n$ agrees with $\gamma$ on $\left[\frac{1}{4}, 1\right]$, we have
\[
\dot{\gamma}_n^\pm(t) \in 4J_0 \partial \HK(\gamma_n(t)), \quad t \in \mathbb{T}.
\]

Therefore, $\gamma_n \in \sys(\partial K)$. It's easy to conclude that
\[
\|\gamma - \gamma_n\|_\infty \leq \frac{1}{n}.
\]

Hence, $\gamma_n$ uniformly converges to $\gamma$. On the other hand, for $t \in [0, \frac{1}{4}]$, it holds that
\[
|\dot{\gamma}(t) - \dot{\gamma}_n(t)| \geq 8,
\]
almost everywhere (except at finitely many points where $\gamma_n$ and $\gamma$ are not differentiable). This implies that $\gamma_n$ does not converge to $\gamma$ in the $W^{1, 1}$-norm.
\end{proof}

From this proposition, we conclude that the spaces $\sys(\partial (B_\infty \times B_1))$ and $\sys_0(\partial (B_\infty \times B_1))$ are not compact in the strong $W^{1,p}$-norm for any $p \in [1, +\infty]$. Hence, from Proposition \ref{uniformtopology}, we conclude that the uniform norm is natural for the spaces $\sys(\partial K)$ and $\sys_0(\partial K)$ in the non-smooth setting.

\subsection{Relationship between spaces of generalized and centralized generalized systoles}

In this subsection, we discuss the relationship between the spaces $\sys(\partial K)$ and $\sys_0(\partial K)$. As we shall see, although these spaces are generally distinct, they coincide from a cohomological point of view. We employ Alexandrov–Spanier cohomology (see \cite{Spa66}), as this is the theory used in the definition of the Fadell–Rabinowitz index in \cite{FR78}.

\begin{proposition}\label{centralizedsystolesconvexfibration}
Let $K \subset \mathbb{R}^{2n}$ be a convex body whose interior contains the origin, and let $\pi_0: \sys(\partial K) \to \sys_0(\partial K)$ denote the orthogonal $L^2$-projection onto the space of centralized generalized systoles.

\begin{enumerate}
    \item The map $\pi_0$ is a continuous, $S^1$-equivariant convex compact fibration. More precisely, $\pi_0$ is a continuous, $S^1$-equivariant surjection such that for every $\gamma_0 \in \sys_0(\partial K)$, the preimage \(\pi_0^{-1}(\gamma_0)\) is a convex and compact set.  
    \item If $K \subset \mathbb{R}^{2n}$ is strictly convex, i.e., for any $x, y \in K$, it holds $[x, y] \setminus \{x, y\}\subset \text{int}(K)$, then $\pi_0$ is an $S^1$-equivariant homeomorphism.   
\end{enumerate}
\end{proposition}

If one drops the strict convexity assumption from statement (2), $\pi_0$ does not have to be injective, even in the smooth category. 

\begin{example}[Examples of convex bodies for which $\pio$ is not injective on the space of systoles]\label{noninjective}

Consider the polydisc 
\[
P(1,1) = \{(z_1, z_2) \in \mathbb{C}^2 \mid \pi |z_1|^2 \leq 1, \ \pi |z_2|^2 \leq 1\}.
\]
The space $\sys(\partial P(1,1))$ consists of loops  
\[
\gamma: \mathbb{T} \to \partial P(1,1), \quad \gamma(t) = \left(e^{2\pi i t} z_1, z_2\right), \quad |z_1|^2 = \frac{1}{\pi}, \ |z_2|^2 \leq \frac{1}{\pi},
\]
and loops  
\[
\gamma: \mathbb{T} \to \partial P(1,1), \quad \gamma(t) = \left(z_1, e^{2\pi i t} z_2\right), \quad |z_1|^2 \leq \frac{1}{\pi}, \ |z_2|^2 = \frac{1}{\pi}.
\]

It is clear that $\pio$ is not injective.

For a smooth example, we can smooth the non-smooth parts of the boundary of $P(1,1)$. More precisely, we consider a perturbation near $\partial D\left(\frac{1}{\pi}\right) \times \partial D\left(\frac{1}{\pi}\right)$.

Note that it holds $c_1^{GH}(\overline{B}(1)) = c_1^{GH}(P(1,1))$, where 
\[
\overline{B}(1) = \{(z_1, z_2) \in \mathbb{C}^2 \mid \pi(|z_1|^2 + |z_2|^2) \leq 1\}.
\]

Therefore, we can choose a small enough perturbation $\widetilde{P}$ of $P(1,1)$ such that $\widetilde{P}$ is a convex body with a smooth boundary, and it holds 
\[
\overline{B}(1) \subset \widetilde{P} \subset P(1,1).
\]

Due to the monotonicity of capacities, it follows that $c_1^{GH}(\widetilde{P}) = c_1^{GH}(\overline{B}(1)) = c_1^{GH}(P(1,1))=1$. 

If we choose a small enough perturbation, we find that for some $a \in (0, \frac{1}{\pi})$, the loops 
\[
\gamma: \mathbb{T} \to \partial \widetilde{P}, \quad \gamma(t) = \left(e^{2\pi i t} z_1, z_2\right), \quad |z_1|^2 = \frac{1}{\pi}, \ |z_2|^2 \leq a,
\]
and   
\[
\gamma: \mathbb{T} \to \partial \widetilde{P}, \quad \gamma(t) = \left(z_1, e^{2\pi i t} z_2\right), \quad |z_1|^2 \leq a, \ |z_2|^2 = \frac{1}{\pi},
\]
are contained in the space $\chara(\partial \widetilde{P})$. Moreover, because their action is $1$, $\widetilde{P}$ is convex, and $c_1^{GH}(\widetilde{P}) = 1$, they belong to $\sys(\partial \widetilde{P})$. Consequently, $\pio$ is not injective on $\sys(\partial \widetilde{P})$, and $\widetilde{P}$ has a smooth boundary.

\end{example}

To prove Proposition \ref{centralizedsystolesconvexfibration}, we will use the following lemma.

\begin{lemma}
The following claims hold:
\begin{enumerate}\label{normalconessubdiferentiallemma}
    \item Let $K \subset \mathbb{R}^m$ be a convex body. If $x, y \in \partial K$ are such that $N_{\partial K}(x) \cap N_{\partial K}(y) \neq \emptyset$, then $[x, y] \subset \partial K$.
    \item Let $f: \mathbb{R}^m \to \mathbb{R}$ be a convex function. Then, for every $x, y \in \mathbb{R}^m$ and every $\alpha \in [0,1]$, it holds that $\partial f(x) \cap \partial f(y) \subseteq \partial f(\alpha x + (1-\alpha)y)$.
\end{enumerate}
\end{lemma}

\begin{proof}
\textbf{Claim} (1): Let's recall the definition of the outward normal cone of $\partial K$ at $x \in \partial K$.  
\[
N_{\partial K}(x) = \{\eta \in \mathbb{R}^{m} \backslash \{0\} \mid \langle \eta, x - z \rangle \geq 0, \quad \forall z \in K\}.
\]

Notice that it makes sense to define $N_{\partial K}(x)$ for $x \in \text{int}(\partial K)$, but such a set will always be empty since for every $\eta \in \mathbb{R}^{m}\backslash\{0\}$, we can find $z \in K$ such that $x - z = \beta \eta$ where $\beta < 0$, which implies that $\langle \eta, x - z \rangle < 0$. Therefore, for $x \in K$, it holds that $x \in \partial K$ if and only if $N_{\partial K}(x) \neq \emptyset$.

Now, if we assume that $\eta \in N_{\partial K}(x) \cap N_{\partial K}(y)$, we have that for every $\alpha \in [0,1]$ and every $z \in K$, it holds:
\[
\langle \eta, \alpha x + (1 - \alpha)y - z \rangle = \alpha\langle \eta, x - z \rangle + (1 - \alpha)\langle \eta, y - z \rangle \geq 0.
\]

Thus, we have shown that for every $\alpha \in [0,1]$, $\eta \in N_{\partial K}(\alpha x + (1-\alpha)y)$. Since $K$ is convex, $\alpha x + (1-\alpha)y \in K$. Therefore, $N_{\partial K}(\alpha x + (1-\alpha)y) \neq \emptyset$, implies that $[x, y] \subset \partial K$.

\textbf{Claim} (2): We recall the definition of a subgradient. Let $x \in \mathbb{R}^m$ be arbitrary.  
\[
\partial f(x) = \{\eta \in \mathbb{R}^{2n} \mid f(z) - f(x) \geq \langle \eta, z - x \rangle, \quad \forall z \in \mathbb{R}^{2n}\}.
\]

Let $\eta \in \partial f(x) \cap \partial f(y)$. Then, for every $\alpha \in [0,1]$ and every $z \in \mathbb{R}^{2n}$, it holds that
\begin{align*}
f(z) - f(\alpha x + (1-\alpha)y) &\geq f(z) - \alpha f(x) - (1-\alpha)f(y) \\
&= \alpha(f(z) - f(x)) + (1-\alpha)(f(z) - f(y)) \\
&\geq \alpha\langle \eta, z - x \rangle + (1-\alpha)\langle \eta, z - y \rangle \\
&= \langle \eta, z - (\alpha x + (1-\alpha)y) \rangle.
\end{align*}

This implies that $\eta \in \partial f(\alpha x + (1-\alpha)y)$ for every $\alpha \in [0,1]$. Hence, the claim holds.

\end{proof}

\begin{proof}[Proof of Proposition \ref{centralizedsystolesconvexfibration}]

\textbf{Claim }(1): Since
\[
\pio(\gamma) = \gamma - \int\limits_{\mathbb{T}} \gamma(t) \, dt
\]

It's clear that $\pio:\sys(\partial K) \to \sys_0(\partial K)$ is continuous with respect to the uniform topology, and $S^1$-equivariant. Surjectivity of this map comes from the fact that $\sys_0(\partial K)$ is precisely defined as the image of the space $\sys(\partial K)$ under the map $\pio$. Since $\pio$ is continuous, $\sys(\partial K)$ is compact, and $\sys_0(\partial K)$ is a $T_1$-space, it follows that for every $\gamma_0 \in \sys_0(\partial K)$, the set $\pio^{-1}(\gamma_0)$ is compact. Therefore, we only need to show that $\pio^{-1}(\{\gamma_0\})$ is convex.

Another way to understand the map $\pio$ is as the map
\[
\gamma \mapsto \dot{\gamma}
\]
where on the space of derivatives we consider the weak$^*$-$L^\infty$ topology. In particular, it holds that $\pio(\gamma_1) = \pio(\gamma_2)$ if and only if $\dot{\gamma_1} = \dot{\gamma_2}$ almost everywhere. We can assume, up to homothety, that the action on the space of generalized systoles equals $1$. Therefore, we have that an absolutely continuous loop $\gamma: \mathbb{T} \to \partial K$ is a systole if and only if
\[
\dot{\gamma}(t) \in J_0\partial \HK(\gamma(t)), \quad \text{a.e.}
\]

Assume now that for $\gamma_1, \gamma_2 \in \sys(\partial K)$, it holds $\pio(\gamma_1) = \pio(\gamma_2) = \gamma_0$, i.e., for almost every $t \in \mathbb{T}$,
\begin{equation}\label{speedintersection}
v_t = \dot{\gamma}_1(t) = \dot{\gamma}_2(t) \in J_0\partial\HK(\gamma_1(t)) \cap J_0\partial\HK(\gamma_2(t)).    
\end{equation}

Since $J_0\mathbb{R}_+\partial \HK(x) = J_0N_{\partial K}(x)$ for every $x \in \partial K$, we conclude that for almost every $t \in \mathbb{T}$,
\[
N_{\partial K}(\gamma_1(t)) \cap N_{\partial K}(\gamma_2(t)) \neq \emptyset
\]

which, by Claim (1) of Lemma \ref{normalconessubdiferentiallemma} and compactness of $\partial K$, implies that for every $t \in \mathbb{T}$, it holds $[\gamma_1(t), \gamma_2(t)] \subset \partial K$. Therefore, for every $\alpha \in [0,1]$, the convex combination $\gamma_\alpha = \alpha\gamma_1 + (1-\alpha)\gamma_2$ is an absolutely continuous function whose image is contained in $\partial K$. From \eqref{speedintersection} and the second statement of Lemma \ref{normalconessubdiferentiallemma}, we have that for almost every $t \in \mathbb{T}$,
\[
\dot{\gamma}_\alpha(t) = \alpha\dot{\gamma}_1(t) + (1-\alpha)\dot{\gamma}_2(t) = v_t \in J_0\partial\HK(\gamma_1(t)) \cap J_0\partial\HK(\gamma_2(t)) \subseteq J_0\partial\HK(\gamma_\alpha(t)).
\]

Hence, for every $\alpha \in [0,1]$, $\gamma_\alpha \in \sys(\partial K)$. Additionally, since $\pio$ is a linear map and $\pio(\gamma_1) = \pio(\gamma_2) = \gamma_0$, it follows that for every $\alpha \in [0,1]$,
\[
\pio(\gamma_\alpha) = \alpha\pio(\gamma_1) + (1-\alpha)\pio(\gamma_2) = \gamma_0.
\]

Thus, we have shown that if $\gamma_1, \gamma_2 \in \pio^{-1}(\{\gamma_0\})$, then $[\gamma_1, \gamma_2] \subseteq \pio^{-1}(\{\gamma_0\})$, which concludes the proof of the first claim.

\textbf{Claim} (2): Again, we can assume that the action on the space of systoles is $1$. If $K \subset \mathbb{R}^{2n}$ is strictly convex, i.e., for every $x, y \in K$, it holds that $[x, y] \backslash \{x, y\} \subset \text{int}(K)$, by claim (1) of Lemma \ref{normalconessubdiferentiallemma}, we conclude that for every $x, y \in \partial K$ such that $x \neq y$, it holds
\[
N_{\partial K}(x) \cap N_{\partial K}(y) = \emptyset.
\]

This, in particular, implies that for every $x, y \in \partial K$, 
\begin{equation}\label{strictlyconvexintersection}
J_0\partial \HK(x)\cap J_0\partial \HK(y)=\emptyset, \quad x\neq y.    
\end{equation}

Assume now that $\gamma_1, \gamma_2 \in \sys(\partial K)$ and $\gamma_1 \neq \gamma_2$. Being that both loops are continuous, there exists an interval $[a, b] \subset \mathbb{T}$ such that $a < b$ and $\gamma_1(t) \neq \gamma_2(t)$ for every $t \in [a, b]$. Since
\[
\dot{\gamma}_{1}(t) \in J_0\partial\HK(\gamma_{1}(t)), \quad \dot{\gamma}_2(t) \in J_0\partial\HK(\gamma_{1}(t)), \quad t \in [a, b] \setminus S
\]
where $S$ is a measure zero set on which $\gamma_1$ and $\gamma_2$ are not differentiable, from \eqref{strictlyconvexintersection}, we conclude that $\dot{\gamma}_{1}(t) \neq \dot{\gamma}_2(t)$ for every $t \in [a, b] \setminus S$. Hence, $\dot{\gamma}_{1}$ and $\dot{\gamma}_2$ differ on a positive measure set, and $\pio(\gamma_1) \neq \pio(\gamma_2)$. Thus, we have shown that $\pio$ is injective, and being that it is surjective by definition, we conclude that the map
\[
\pio: \sys(\partial K) \to \sys_0(\partial K)
\]
is a continuous $S^1$-equivariant bijection. Since $\sys(\partial K)$ is compact and $\sys_0(\partial K)$ is Hausdorff, the conclusion follows.

\end{proof}

\begin{proposition}\label{homologycalequivalanceofsystolesandcenteralizedsystoles}
    Let $K \subset \mathbb{R}^{2n}$ be a convex body whose interior contains the origin, and let $\pi_0: \sys(\partial K) \to \sys_0(\partial K)$ denote the orthogonal $L^2$-projection. Then the following statements hold:

\begin{enumerate}
    \item The map $\pi_0$ induces an isomorphism in cohomology.
    \item The map $\pi_0$ induces an isomorphism in $S^1$-equivariant cohomology.
    \item The spaces $\sys(\partial K)$ and $\sys_0(\partial K)$ have the same Fadell–Rabinowitz index, that is,
    \[
    \indfr(\sys(\partial K)) = \indfr(\sys_0(\partial K)).
    \]
\end{enumerate}

\end{proposition}

\begin{proof}

\textbf{Claim }(1): Since $\sys(\partial K)$ is compact and $\sys_0(\partial K)$ is Hausdorff, we conclude that $\pio: \sys(\partial K) \to \sys_0(\partial K)$ is a closed map. Moreover, $\pio$ is surjective, and by statement (1) of Proposition \ref{centralizedsystolesconvexfibration}, for every $\gamma_0$, the fiber $\pio^{-1}(\gamma_0)$ is convex and therefore acyclic. The conclusion then follows from \cite[Theorem 4.2]{Med13}.

\textbf{Claim }(2): We need to show that the induced map
\[
\pi_0^*: H^*_{S^{1}}(\sys_0(\partial K)) \to H^*_{S^{1}}(\sys(\partial K))
\]
in $S^1$-equivariant cohomology is an isomorphism. We will use an argument analogous to the one in the previous claim.

Note that $\pi_0^* = (\widetilde{\pi}_0)^*$, where
\[
\widetilde{\pi}_0: \sys(\partial K) \times_{S^1} ES^1 \to \sys_0(\partial K) \times_{S^1} ES^1, \quad 
\widetilde{\pi}_0([\gamma, z]) = [\pi_0(\gamma), z],
\]
and
\[
(\widetilde{\pi}_0)^*: H^*(\sys_0(\partial K) \times_{S^1} ES^1) \to H^*(\sys(\partial K) \times_{S^1} ES^1).
\]
The map $\widetilde{\pi}_0$ is well defined since $\pi_0$ is $S^1$-equivariant. Thus, to show that $\pi_0^* = (\widetilde{\pi}_0)^*$ induces an isomorphism in cohomology, it suffices to show that $\widetilde{\pi}_0$ has acyclic fibers and is closed. As in the previous claim, \cite[Theorem 4.2]{Med13} then implies the result.

Let $[\gamma_0, z] \in \sys_0(\partial K) \times_{S^1} ES^1$ be arbitrary, and fix $(\gamma_0, z)$ as a representative of this class. Clearly,
\[
\widetilde{\pi}_0^{-1}([\gamma_0, z]) = \{[\gamma, z] \mid \gamma \in \pi_0^{-1}(\gamma_0)\}.
\]
Since $\pi_0^{-1}(\gamma_0)$ is convex by statement (1) of Proposition \ref{centralizedsystolesconvexfibration}, it follows that $\widetilde{\pi}_0^{-1}([\gamma_0, z])$ is contractible and therefore acyclic. It remains to show that $\widetilde{\pi}_0$ is a closed map.

Consider the map
\[
\overline{\pi}_0: \sys(\partial K) \times ES^1 \to \sys_0(\partial K) \times_{S^1} ES^1, \quad 
\overline{\pi}_0(\gamma, z) = [\pi_0(\gamma), z].
\]
This map induces $\widetilde{\pi}_0$ on the orbit space of $\sys(\partial K) \times ES^1$. We first show that $\overline{\pi}_0$ is closed. This follows because it is a composition of two closed maps. Indeed,
\[
\overline{\pi}_0 = p_0 \circ (\pi_0 \times id_{ES^1}),
\]
where
\[
p_0: \sys_0(\partial K) \times ES^1 \to \sys_0(\partial K) \times_{S^1} ES^1
\]
is the quotient map associated with the group action, and
\[
\pi_0 \times id_{ES^1}: \sys(\partial K) \times ES^1 \to \sys_0(\partial K) \times ES^1
\]
is the corresponding product map. The map $p_0$ is closed since $S^1$ acts continuously on $\sys_0(\partial K) \times ES^1$, and $\pi_0 \times id_{ES^1}$ is closed as a perfect map, because the product of perfect maps is perfect (see \cite[Theorem 3.7.9]{Eng89}). Therefore, $\overline{\pi}_0$ is closed.

Now denote by
\[
p: \sys(\partial K) \times ES^1 \to \sys(\partial K) \times_{S^1} ES^1
\]
the quotient map associated with the diagonal $S^1$-action on $\sys(\partial K) \times ES^1$.

Let $A \subset \sys(\partial K) \times_{S^1} ES^1$ be an arbitrary closed set. Since $p$ is a quotient map, we have
\[
A \subseteq \sys(\partial K) \times_{S^1} ES^1 \text{ is closed } \iff p^{-1}(A) \subseteq \sys(\partial K) \times ES^1 \text{ is closed.}
\]
It is clear that
\[
\widetilde{\pi}_0(A) = \overline{\pi}_0(p^{-1}(A)),
\]
and since $p^{-1}(A)$ is closed and $\overline{\pi}_0$ is a closed map, it follows that $\widetilde{\pi}_0(A)$ is closed in $\sys_0(\partial K) \times_{S^1} ES^1$. Therefore, $\widetilde{\pi}_0$ is closed, which completes the proof of the claim.

\textbf{Claim} (3): This claim follows directly from the previous one. Indeed, consider the maps
\[
\pi_2: \sys(\partial K) \times_{S^1} ES^1 \to BS^{1}, \quad 
\pi_2^0: \sys_0(\partial K) \times_{S^1} ES^1 \to BS^1,
\]
which are the projections onto the second coordinate.

Note that
\[
\pi_2^* = \pi_0^* \circ (\pi_2^0)^*.
\]
Since $\pi_0^*$ is an isomorphism, it is clear that
\[
\pi_2^* e^i \neq 0 \iff (\pi_2^0)^* e^i \neq 0,
\]
which, by the definition of the Fadell–Rabinowitz index, implies that
\[
\indfr(\sys(\partial K)) = \indfr(\sys_0(\partial K)).
\]

\end{proof}

\section{Systolic $S^1$-index of convex bodies}

In this section, we will present the proof of Theorem \ref{mainsystolics1index}, as well as the proof of the upper-semicontinuity property and the existence of a uniform bound for $\indsys$. These proofs are provided in Subsection \ref{subsectionmainsystolics1index}. To prove Theorem \ref{mainsystolics1index}, we will use Clarke's duality (see \cite{Cla81}). 

\subsection{Clarke's duality}

Let us assume that the origin lies in the interior of a convex body $K$. We can define a positively $2$-homogeneous function $\HK:\mathbb{R}^{2n}\to \mathbb{R}$ such that $\HK^{-1}(1) = \partial K$. Since $\HK$ is convex, we can define its Fenchel conjugate, denoted by $\HdK$, as follows:
\[
\HdK: \mathbb{R}^{2n} \to \mathbb{R}, \quad \HdK(x) = \max_{y \in \mathbb{R}^{2n}} \big( \langle x, y \rangle - \HK(y) \big).
\]

Consider the Sobolev space
\[
\Honenull = \left\{ x \in H^1(\mathbb{T}; \mathbb{R}^{2n}) \mid \int_\mathbb{T} x(t) \, dt = 0 \right\},
\]
equipped with the norm $\|x\|_{H^1_0} = \|\dot{x}\|_{L^2}$. This norm is equivalent to the standard Sobolev $H^1$-norm on this space. On $\Honenull$, we have a natural $S^1$-action given by 
\[
\theta \cdot x = x(\cdot - \theta), \quad \theta \in \mathbb{T}, \ x \in \Honenull.
\]

We define the functionals
\[
\mathcal{A}: \Honenull \to \mathbb{R}, \quad \mathcal{A}(x) = \frac{1}{2} \int_\mathbb{T} \langle \dot{x}(t), J_0 x(t) \rangle \, dt,
\]
and
\[
\mathcal{H}_{K}: \Honenull \to \mathbb{R}, \quad \mathcal{H}_{K}(x) = \int_\mathbb{T} \HdK(-J_0 \dot{x}(t)) \, dt.
\]

Notice that these functionals are $S^1$-invariant. The Clarke's dual functional associated with $K$ is defined as
\[
\widetilde{\Psi}_K: \{\mathcal{A} > 0\} \to \mathbb{R}, \quad \widetilde{\Psi}_K(x) = \frac{\mathcal{H}_{K}(x)}{\mathcal{A}(x)}.
\]
This functional is $0$-homogeneous and, therefore, $\mathbb{C}_{*} = \mathbb{C} \setminus \{0\}$-invariant. It can be restricted to various types of $S^1$-invariant hypersurfaces. Here, we choose $\Lambda = \mathcal{A}^{-1}(1)$ and define the restriction of this dual functional as
\[
\psiK: \Lambda \to \mathbb{R}, \quad \psiK := \widetilde{\Psi}_K|_{\Lambda}=\mathcal{H}_K|_\Lambda.
\]

To introduce the Ekeland-Hofer spectral invariants from \cite{EH87}, we will use the Fadell-Rabinowitz index defined in the introduction. The Fadell-Rabinowitz index\footnote{In this definition, the Fadell-Rabinowitz index is shifted by 1, aligning with the definition provided in \cite[Section 5]{FR78}.} was first introduced in \cite{FR78}, where its properties were also explored. In our case, the field $\mathbb{F}$ will be $\mathbb{Z}_2$.

\textbf{Spectral invariants}: 
Let $K \subset \mathbb{R}^{2n}$ be a convex body whose interior contains the origin. 
We define the \textit{$i$-th Ekeland–Hofer spectral invariant} of $K$ as
\[
s_i(K) := \inf \{L > 0 \mid \indfr(\{\psiK < L\}) \geq i\}, \quad i \in \mathbb{N},
\]
where the topology on $\{\psiK < L\}$ is induced by the $H^1_0$-norm.

We now introduce an analogous sequence of invariants defined with respect to a different topology.

\textbf{Weak spectral invariants}: 
We define the \textit{$i$-th weak Ekeland–Hofer spectral invariant} as
\[
s_i^{\infty}(K) := \inf \{L > 0 \mid \indfr(\{\psiK < L\}) \geq i\}, \quad i \in \mathbb{N},
\]
where the topology on $\{\psiK < L\}$ is induced by the uniform norm.

\begin{remark}
The terminology “weak Ekeland–Hofer spectral invariants” is justified as follows. 
Let us define the sequence  
\[
s_i^{\operatorname{w}}(K) := \inf \{L > 0 \mid \indfr(\{\psiK < L\}) \geq i\}, \quad i \in \mathbb{N},
\]
where the topology on $\{\psiK < L\}$ is induced by the weak-$H^1$ topology. 
By statement (1) of Lemma~\ref{compactsublevelsclarke}, on closed sublevel sets the weak topology is equivalent to the uniform topology. 
Hence, for every convex body $K \subset \mathbb{R}^{2n}$ and every $i \in \mathbb{N}$, we have  
\[
s_i^{\operatorname{w}}(K) = s_i^{\infty}(K).
\]

\end{remark}

Both maps $K \mapsto s_i(K)$ and $K \mapsto s_i^{\infty}(K)$ are monotone with respect to inclusions of convex bodies, and they are positively $2$-homogeneous. 
This implies that they are continuous with respect to the Hausdorff distance topology (see \cite{EH87} for $s_i$).

\begin{lemma}\label{clarkeLinfinity}
Let $K \subset \mathbb{R}^{2n}$ be a convex body whose interior contains the origin. 
Then, for all $i \in \mathbb{N}$, we have
\[
s_i(K) = s_i^{\infty}(K).
\]
\end{lemma}

To prove this lemma, we will use the continuity of $s_i$ and $s_i^{\infty}$, together with the finite-dimensional reduction for Clarke’s dual functional.

\textbf{Finite-Dimensional Reduction}

We say that $K$ is a smooth and strongly convex body if it is a convex body with a smooth boundary that has positive sectional curvature everywhere. In this case, $\psiK$ is $C^{1,1}$.

Every element $x \in \Honenull$ can be represented as 
\[
x(t) = \sum_{k \in \mathbb{Z} \setminus \{0\}} e^{2\pi k t J_0} \hat{x}(k),
\]
where $\hat{x}: \mathbb{Z} \setminus \{0\} \to \mathbb{R}^{2n}$ satisfies
\[
\sum_{k \in \mathbb{Z} \setminus \{0\}} |k|^2 |\hat{x}(k)|^2 < \infty.
\]

Thus, for every $N \in \mathbb{N}$, the space $\Honenull$ can be decomposed as
\[
\Honenull = \mathbb{H}_N \oplus \mathbb{H}^N,
\]
where 
\[
\mathbb{H}_N = \{ x \in \mathbb{H}^1_0 \mid \hat{x}(k) = 0 \text{ if } k < -N \text{ or } k > N \}
\]
and
\[
\mathbb{H}^N = \{ x \in \mathbb{H}^1_0 \mid \hat{x}(k) = 0 \text{ if } -N \leq k \leq N \}.
\]

This decomposition is orthogonal. We denote the orthogonal projection onto $\mathbb{H}_N$ by
\[
\mathbb{P}_N : \Honenull \to \mathbb{H}_N.
\]

Let
\[
A_N = \{ x \in \{\mathcal{A} > 0\} \mid \mathbb{P}_N(x) \neq 0 \},
\]

and let $b > \min \psiK$ be arbitrary. We define
\[
V_N = \mathbb{P}_N(\{\widetilde{\Psi}_K < b\}).
\]

and for every $x \in V_N$, we define the space
\[
W_x = \{ y \in \mathbb{H}^N \mid \mathcal{A}(x+y) > 0, \ \widetilde{\Psi}_K(x+y) < b \},
\]
which is clearly non-empty. For $N \in \mathbb{N}$ large enough the following claims hold.

\begin{itemize}
    \item $\{\widetilde{\Psi}_K < b\} \subset A_N$.
    \item For every $x \in V_N$, the function
    \[
    W_x \to \mathbb{R}, \quad y \mapsto \widetilde{\Psi}_K(x+y)
    \]
    has a unique global minimizer denoted by $\eta(x)$.
    \item The function $\eta: V_N \to \mathbb{H}^N$ is an $\mathbb{C}_*$-equivariant $C^{1,1}$ function.
\end{itemize}

To see that $\{\widetilde{\Psi}_K < b\} \subset A_N$ indeed holds for $N$ large enough, see the proof of \cite[Lemma 3.4]{BBLM23}. Two statements that followed are corollaries of the work \cite{EH87}.
Indeed, the authors constructed an $S^1$-invariant $C^{1,1}$-function 
\[
\eta: U_N \to \mathbb{H}^N,
\]
where $U_N = V_N \cap S$, $S$ denotes the unit sphere in $\mathbb{H}_N$ with respect to the $H^1$-norm, and $\eta(x)$ is the unique global minimizer of the function 
\[
W_x \to \mathbb{R}, \quad y \mapsto \widetilde{\Psi}_K(x + y).
\]

Since $\widetilde{\Psi}_K$ is $0$-homogeneous, we can take a $1$-homogeneous extension of $\eta$, which gives us the previously described function. This construction provides the freedom to choose $S^1$-invariant hypersurfaces of $\mathbb{H}_N$ and $\{\mathcal{A} > 0\}$.

We define the reduced Clarke's dual functional as 
\[
\psi_K: U_N \to \mathbb{R}, \quad \psi_K(x) = \widetilde{\Psi}_K(x + \eta(x)) = \psiK(i(x)),
\]
where
\[
i(x) = \frac{x + \eta(x)}{\sqrt{\mathcal{A}(x + \eta(x))}}.
\]

For such a reduction, it holds that, for every $L \in (0, b)$, the map
\[
i_L: \{\psi_K < L\} \to \{\psiK < L\}, \quad i_L(x) = i(x),
\]
is an $S^1$-equivariant homotopy equivalence, where 
\[
p_L: \{\psiK < L\} \to \{\psi_K < L\}, \quad p_L(x) = \frac{\mathbb{P}_N(x)}{\|\mathbb{P}_N(x)\|_{H^1_0}},
\]
is the homotopy inverse (see \cite{EH87}).

\begin{proof}[Proof of Lemma \ref{clarkeLinfinity}]

Let $K \subset \mathbb{R}^{2n}$ be a smooth and strongly convex body whose interior contains the origin. Let $L > \min \sigma(\partial K) = \min \psiK$ be arbitrary. Due to $S^1$-equivariant continuous embeddings, we have
\begin{equation}\label{frh1linfl2}
    \indfr((\{\psiK < L\}, \|\cdot\|_{H^1_0})) \leq \indfr((\{\psiK < L\}, \|\cdot\|_{\infty})) \leq \indfr((\{\psiK < L\}, \|\cdot\|_{L^2})).
\end{equation}

We choose $b > L$ and $N \in \mathbb{N}$ large enough such that the reduced Clarke's dual functional $\psi_K$ exists. Since $i_L$ is a $S^1$-equivariant homotopy equivalence, we have
\begin{equation}\label{frhomeq}
    \indfr((\{\psi_K < L\}, \|\cdot\|_{H^1_0})) = \indfr((\{\psiK < L\}, \|\cdot\|_{H^1_0})).
\end{equation}

On the other hand, the orthogonal $L^2$-projection
\[
\mathbb{P}_N^{L^2}: (L^2(\mathbb{T}, \mathbb{R}^{2n}), \|\cdot\|_{L^2}) \to (\mathbb{H}_N, \|\cdot\|_{L^2})
\]
is continuous. Since $\mathbb{H}_N$ is finite-dimensional, we can change the norm in the codomain to the $H^1$-norm. Therefore, we conclude that the map
\[
p^{L^2}_L: (\{\psiK < L\}, \|\cdot\|_{L^2}) \to (\{\psi_K < L\}, \|\cdot\|_{H^1_0}), \quad p^{L^2}_L(x) = \frac{\mathbb{P}_N(x)}{\|\mathbb{P}_N(x)\|_{H^1_0}}
\]
is a continuous $S^1$-equivariant map. Therefore, it follows that
\begin{equation}\label{frl2h1}
    \indfr((\{\psiK < L\}, \|\cdot\|_{L^2})) \leq \indfr((\{\psi_K < L\}, \|\cdot\|_{H^1_0})).
\end{equation}

Now, combining \eqref{frh1linfl2}, \eqref{frhomeq}, and \eqref{frl2h1}, we deduce
\[
\indfr((\{\psiK < L\}, \|\cdot\|_{H^1_0})) = \indfr((\{\psiK < L\}, \|\cdot\|_{\infty})).
\]

Since $L > \min \psiK$ was arbitrary, we conclude that $s_i(K) = s_i^\infty(K)$ holds. Finally, since smooth  strongly convex bodies are dense in convex bodies, and $s_i$ and $s_i^\infty$ are continuous in the Hausdorff-distance topology, the lemma follows.

\end{proof}

\textbf{Weak critical points of $\psiK$:} We say that $x \in \Lambda$ is a weak critical point of $\psiK$ if there exist constants $\alpha_1, \alpha_2 \in \mathbb{R}$, not both equal to zero, such that

\[
0 \in \partial (\alpha_1 \mathcal{H}_K(x) + \alpha_2 \mathcal{A}(x)).
\]

We denote the set of weak critical points of $\psiK$ by $\text{crit}(\psiK)$.

\begin{lemma}\label{bijecthom}
    A point $x \in \Lambda$ is a weak critical point of $\psiK$ if and only if
    \[
    \psiK(x) x(t) + \beta \in \partial \HdK(-J_0 \dot{x}(t)),
    \]
    where $\beta \in \mathbb{R}^{2n}$ is a constant vector. 

    Moreover, there exists a surjective map
    \[
    \mathcal{P}: \chara(\partial K) \to \crit(\psiK), \quad \mathcal{P}(\gamma) = \frac{1}{\sqrt{\mathcal{A}(\gamma)}} \pi_0(y),
    \]
    where $\pi_0(\gamma) = \gamma - \int_\mathbb{T} \gamma(t) \, dt$.

    Additionally, it holds that
    \[
    \mathcal{A}(\gamma) = \psiK(\mathcal{P}(\gamma)),
    \]
    for every $\gamma \in \chara(\partial K)$.
\end{lemma}

The proof of this lemma can be found in \cite{Cla81, AO14}.

\begin{remark}\label{convexfibration}
   In the case of a smooth and strongly convex body, $\mathcal{P}$ is a bijection (see \cite{Cla79, HZ94, AO08}). In fact, one can show that $\mathcal{P}$ is bijective under weaker assumptions, namely if $K$ is strictly convex (see the proof of the second statement of Proposition \ref{centralizedsystolesconvexfibration}). However, in the general case, $\mathcal{P}$ is surjective but not necessarily injective, even in the smooth case (see Example \ref{noninjective} of convex bodies $K$ for which $\pio$ and, therefore, $\mathcal{P}$ is not injective on $\sys(\partial K)$). Nevertheless, for each $\beta \in \mathbb{R}^{2n}$ such that
\[
\psiK(x)x(t) + \beta \in \partial \HdK(-J_0 \dot{x}(t)),
\]
there exists a unique corresponding generalized closed characteristic given by
\[
\gamma = \frac{1}{\sqrt{\psiK(x)}} (\psiK(x)x + \beta),
\]
and it holds that \[\mathcal{A}(\gamma) = \psiK(x).\]

Moreover, the fibers of $\mathcal{P}$ are convex and compact. This property follows from the fact that $\partial \HdK(y)$ is a convex compact subset of $\mathbb{R}^{2n}$ for each $y \in \mathbb{R}^{2n}$.

\end{remark}

\begin{lemma}\label{compactsublevelsclarke} 
    The following statements hold:
    \begin{enumerate}
        \item The space of centralized systoles, $\sys_0(\partial K)$, and the set of minimum points of $\psiK$, $\sys_*(\partial K)$, are $S^1$-homeomorphic spaces. In particular, it holds:
        \[
        \indfr(\sys_0(\partial K)) = \indfr(\sys_*(\partial K)).
        \]
        \item Closed sublevels of $\psiK$ are weakly sequentially compact in $H^1_0$ and hence strongly compact in the uniform norm.    
    \end{enumerate}
\end{lemma}

\begin{proof}
\textbf{Claim }(1): By the definition of $\sys_0(\partial K)$, we have that  
\begin{equation}\label{sysosys}
\sys_0(\partial K) = \pio(\sys(\partial K)),    
\end{equation}

where $\pi_0(y) = y - \int_\mathbb{T} y(t) \, dt$.

On the other hand, the minimum points of $\psiK$ must be weak critical points of $\psiK$, and therefore Lemma \ref{bijecthom} implies that there is a surjection from generalized systoles, $\sys(\partial K)$, to the set of minimal points of $\psiK$, $\sys_*(\partial K)$, given by

\[
\mathcal{P}: \sys(\partial K) \to \sys_*(\partial K), \quad \mathcal{P}(\gamma) = \frac{1}{\sqrt{T_{\min}}} \pi_0(\gamma).
\]

where $T_{\min} > 0$ is the action on $\sys(\partial K)$. Hence,

\begin{equation}\label{sys*sys}
\sys_*(\partial K) = \frac{1}{\sqrt{T_{\min}}} \pio(\sys(\partial K)).    
\end{equation}

Combining \eqref{sysosys} and \eqref{sys*sys}, we conclude that  
\[
\sys_*(\partial K) = \frac{1}{\sqrt{T_{\min}}} \sys_0(\partial K),
\]

and the claim obviously holds.

\textbf{Claim }(2): Here, we follow arguments from \cite{AO08, AO14}. Let $x_i \in \{\psiK \leq M\}$ be an arbitrary sequence, where $M > 0$ is any positive constant. For some $\beta > 1$, it holds that 
\[
\frac{1}{\beta} |x|^2 \leq \HdK(x).
\]

From this, it follows that 
\begin{equation}\label{uniformbound}
\|x_i\|_{H^1_0} \leq \beta \psiK(x_i) \leq \beta M.
\end{equation}

Therefore, $x_i$ is bounded in $\Honenull$ and, hence, up to passing to a subsequence, weakly convergent in $\Honenull$. Let $x_*$ denote its weak limit. Moreover, $x_i$ converges to $x_*$ in the uniform norm due to the Arzelà–Ascoli theorem. Indeed, from \eqref{uniformbound}, it follows that
\[
|x_i(t) - x_i(s)| = \left| \int_{s}^{t} \dot{x}_i(\tau) \, d\tau \right| \leq (\beta M)^{\frac{1}{2}} |s - t|^{\frac{1}{2}}.
\]

Now, we show that $x_* \in \Lambda$. Splitting the term 
\[
\frac{1}{2} \int_{\mathbb{T}} \langle \dot{x}_i(t), J_0 x_i(t) \rangle \, dt
\]
as 
\[
1 = \mathcal{A}(x_i) = \frac{1}{2} \int_{\mathbb{T}} \langle \dot{x}_i(t), J_0 x_i(t) \rangle \, dt = \frac{1}{2} \int_{\mathbb{T}} \langle \dot{x}_i(t), J_0 (x_i(t) - x_*(t)) \rangle \, dt + \frac{1}{2} \int_{\mathbb{T}} \langle \dot{x}_i(t), J_0 x_*(t) \rangle \, dt
\]
and taking limit as $i \to \infty$ yields $\mathcal{A}(x_*) = 1$, so $x_* \in \Lambda$.

Finally, we prove that $x_* \in \{\psiK \leq M\}$. From the convexity of $\HdK$, we have 
\begin{equation}\label{pointwise}
\HdK(-J_0\dot{x}_*(t)) - \HdK(-J_0\dot{x}_i(t)) \leq \langle \partial \HdK(-J_0\dot{x}_*(t)), J_0(\dot{x}_i(t) - \dot{x}_*(t)) \rangle,
\end{equation}
where the inequality holds in the set-theoretic sense (for any element in the corresponding subdifferential).

To prove that $x_* \in \{\psiK \leq M\}$, we need to show that there exists a measurable $L^2$-section of $\partial \HdK(-J_0\dot{x}_*(t))$, i.e., a measurable function $s:\mathbb{T} \to \mathbb{R}^{2n}$ such that $s(t) \in \partial \HdK(-J_0\dot{x}_*(t))$ almost everywhere, and $s \in L^2(\mathbb{T}, \mathbb{R}^{2n})$. Since $|\partial \HdK(y)| \leq c|y|$ for all $y$, for some positive constant $c$, this implies that if the section exists, it must belong to $L^2$, as $\dot{x}_* \in L^2(\mathbb{T}, \mathbb{R}^{2n})$. A standard measure-theoretic argument ensures the existence of such a section, which we denote by $s$. For this choice, from \eqref{pointwise}, we have  
\[
\psiK(x_*) - \psiK(x_i) \leq \int_{\mathbb{T}} \langle s(t), J_0(\dot{x}_i(t) - \dot{x}_*(t)) \rangle \, dt.
\]

The right-hand side of this inequality converges to $0$ as $i \to \infty$ (due to weak convergence), which implies 
\[
\psiK(x_*) \leq \liminf_{i \to \infty} \psiK(x_i) \leq M.
\]
Thus, $x_* \in \{\psiK \leq M\}$, which concludes the proof.

\end{proof}

We now have all the ingredients needed to prove Theorem~\ref{mainsystolics1index}.

\subsection{Properties of the systolic $S^1$-index}\label{subsectionmainsystolics1index}

\begin{proof}[Proof of Theorem \ref{mainsystolics1index}]

Let $K \subset \mathbb{R}^{2n}$ be a convex body whose interior contains the origin. We need to show that 
\[
\indfr(\sys(\partial K), \|\cdot\|_{\infty}) = \max\{i \mid c^{GH}_i(K) = c^{GH}_1(K)\}.
\]

From \cite[Theorem A]{Mat24}, we have 
\[
c^{GH}_i(K) = s_i(K), \quad k \in \mathbb{N},
\]
and by Lemma \ref{clarkeLinfinity}, it follows that 
\[
s_i(K) = s^{\infty}_i(K), \quad k \in \mathbb{N}.
\]
Furthermore, by statement (3) of Proposition \ref{homologycalequivalanceofsystolesandcenteralizedsystoles}, we have 
\[
\indfr(\sys(\partial K)) = \indfr(\sys_0(\partial K)),
\]
and by claim (1) of Lemma \ref{compactsublevelsclarke}, it follows that 
\[
\indfr(\sys_0(\partial K)) = \indfr(\sys_*(\partial K)),
\]
where $\sys_*(\partial K)$ denotes the set of minimum points of $\psiK$. Thus, it remains to prove that 
\[
\indfr(\sys_*(\partial K), \|\cdot\|_{\infty}) = \max\{i \mid s^{\infty}_i(K) = s^{\infty}_1(K)\}.
\]

Assume that for some $i \geq 2$ it holds that $s^\infty_i(K) > s^\infty_1(K) = \min \psiK$. By the definition of $s^\infty_i$ and the monotonicity of the index, for any $L \in (s^\infty_1(K), s^\infty_i(K))$, we have 
\[
\indfr(\{\psiK < L\}) < i.
\]
Since $\sys_*(\partial K) \subset \{\psiK < L\}$, it follows that $\indfr(\sys_*(\partial K)) < i$. 

Conversely, if $s^\infty_i(K) = s^\infty_1(K)$, then 
\begin{equation}\label{biggerthank}
\indfr(\{\psiK < L\}) \geq i, \quad \forall L > \min \psiK.
\end{equation}
Since $\sys_*(\partial K)$ is compact by claim (2) of Lemma \ref{compactsublevelsclarke}, and $(C^0(\mathbb{T}, \mathbb{R}^{2n}), \|\cdot\|_\infty)$ is paracompact, there exists an $S^1$-invariant neighborhood $U$ of $\sys_*(\partial K)$ in $C^0(\mathbb{T}, \mathbb{R}^{2n})$ such that 
\begin{equation}\label{compactneighbourhood}
\indfr(\sys_*(\partial K)) = \indfr(U).
\end{equation}
For details, see \cite{FR78}. 

By claim (2) of Lemma \ref{compactsublevelsclarke}, for sufficiently small $L > \min \psiK$, it holds that $\{\psiK < L\} \subseteq U$. Therefore, 
\begin{equation}\label{inequalitysublevel}
\indfr(\{\psiK < L\}) \leq \indfr(U),
\end{equation}
by the monotonicity property of the Fadell–Rabinowitz index. From \eqref{biggerthank}, \eqref{inequalitysublevel}, and \eqref{compactneighbourhood}, it follows that 
\[
\indfr(\sys_*(\partial K), \|\cdot\|_{\infty}) \geq i,
\]
which completes the proof.

\end{proof}

Let $\conv$ denote the set of all convex bodies endowed with the Hausdorff-distance topology. Theorem~\ref{mainsystolics1index} implies that $\indsys(K)$ is well-defined (see Definition \ref{definitionofthesystolicS1index} of $\indsys(K)$) and finite since $c_i^{GH}(K) \to \infty$ as $i \to \infty$. Therefore, the function 
\[
\indsys: \conv \to \mathbb{N}
\]
is well-defined. Moreover, this function satisfies the following properties.

\begin{proposition}\label{uppersemicontuperbound}
The following statements hold:
\begin{enumerate}
    \item The function
    \[
    \indsys: \conv \to \mathbb{N}
    \]
    is upper semi-continuous.
    \item For every $K \in \conv$, we have $\indsys(K) \leq 4n^3$. If $K = -K$, then $\indsys(K) \leq 2n^2$.
    \item For every $S^1$-invariant convex body $K \subset \mathbb{R}^{2n}$, where the $S^1$-action is the standard one, it holds that $\indsys(K) \leq n$.
\end{enumerate}
\end{proposition}

\begin{proof}
\textbf{Claim (1):} Let $\indsys(K) = i$. This means that $c^{GH}_{i+1}(K) > c^{GH}_1(K)$. By the continuity of the Gutt-Hutchings capacities, there exists a Hausdorff-distance neighborhood $U \subset \conv$ of $K$ such that for every $K' \in U$, we have $c^{GH}_{i+1}(K') > c^{GH}_1(K')$, which, by Theorem~\ref{mainsystolics1index}, implies that $\indsys(K') \leq i$. This completes the proof of the first claim.

\textbf{Claim (2):} For an arbitrary $K \in \conv$, there exists an ellipsoid $E$ of minimal volume containing $K$ such that 
\begin{equation}\label{johnelipsoid}
\frac{1}{2n}E \subseteq K \subseteq E.
\end{equation}

The existence of such an ellipsoid, known as the Loewner-Behrend-John ellipsoid, is proven in \cite{Joh48} and \cite[Appendix B]{Vit00}.

For $a_1,\dots,a_n > 0$, we define the sequence $M_i(a_1,\dots,a_n)$ of positive integer multiples of $a_1,\dots,a_n$, arranged in non-decreasing
order with repetitions. For an ellipsoid 
\[
E(a_1,\dots,a_n) = \left\{ z \in \mathbb{C}^n \mid \sum\limits_{i=1}^n \frac{\pi |z_i|^2}{a_i} \leq 1 \right\},
\]
where $a_1,\dots,a_n > 0$, it holds that $c^{GH}_i(E(a_1,\dots,a_n)) = M_i(a_1,\dots,a_n)$ (see \cite{GH18}). Since $M_{ni+1}(a_1,\dots,a_n)$ must be at least the $(i+1)$-th multiple of some of the numbers and $\allowbreak M_1(a_1,\dots,a_n) = \min\{a_1,\dots,a_n\}$, we conclude that it holds $c_{ni+1}^{GH}(E) \geq (i+1) c^{GH}_1(E)$.

From \eqref{johnelipsoid}, we know that 
\[
c_1^{GH}(E) \geq c_1^{GH}(K)
\]
and 
\[
c^{GH}_{4n^3+1}(K) \geq c^{GH}_{4n^3+1}\left(\frac{1}{2n}E\right) = \frac{1}{4n^2}c^{GH}_{4n^3+1}(E) \geq \frac{1}{4n^2}(4n^2+1)c^{GH}_1(E) > c_1^{GH}(E).
\]
Combining these estimates, we obtain $c^{GH}_{4n^3+1}(K) > c^{GH}_1(K)$, which, by Theorem~\ref{mainsystolics1index}, implies that $\indsys(K) \leq 4n^3$.

If $K = -K$, then for the Loewner-Behrend-John ellipsoid, we have 
\[
\frac{1}{\sqrt{2n}}E \subseteq K \subseteq E,
\]
which, by the same methods, implies that $\indsys(K) \leq 2n^2$. This completes the proof of the second claim.

\textbf{Claim (3):} For the standard $S^1$-action given by
\[
\theta \cdot z = e^{2\pi i \theta} z, \quad \theta \in \mathbb{T}, \ z \in \mathbb{C}^n,
\]
we have that, for an $S^1$-invariant convex body $K \subset \mathbb{R}^{2n}$, $c_1^{GH}(K)$ coincides with the Gromov width of $K$ (see \cite{GHR22}).

Therefore, we can find a ball 
\[
B(r) = \{z \in \mathbb{C}^n \mid \pi \|z\| < r\}
\]
that symplectically embeds into $K$, with 
\[
r \in \left(\frac{c^{GH}_1(K)}{2}, \, c^{GH}_1(K)\right).
\]

Since $B(r)$ symplectically embeds into $K$, it follows that 
\[
c^{GH}_{n+1}(K) \geq c^{GH}_{n+1}(\overline{B}(r)) = 2r > c^{GH}_1(K).
\]

By Theorem~\ref{mainsystolics1index}, this implies that 
\[
\indsys(K) \leq n.
\]

\end{proof}

\section{Generalized Zoll convex bodies}

In this section, we present the proof of Theorem \ref{maingeneralizedzoll}, along with its corollary on the upper bound of $\indsys$ in the smooth case. We also discuss concrete examples and explore the relationship between the generalized Zoll property and the evaluation map on the space of systoles.

\subsection{Properties of generalized Zoll convex bodies}

\begin{proof}[Proof of Theorem \ref{maingeneralizedzoll}]

\textbf{Claim (1):} The claim follows immediately from Theorem \ref{mainsystolics1index} and the definition of generalized Zoll convex bodies.

\textbf{Claim} (2): We follow a modified version of the arguments in \cite[Lemma~3.1]{GGM21}. This modification is necessary because, in our setting, no regularity is assumed on the boundary, and thus the construction of a Riemannian metric used in the original proof is not applicable.

Let $K \subset \mathbb{R}^{2n}$ be a convex body whose interior contains the origin. We consider the evaluation map
\[
\ev : \sys(\partial K) \to \partial K, \quad \ev(\gamma) = \gamma(0).
\]

The condition that $K$ satisfies the uniqueness of systoles property is equivalent to the injectivity of $\ev$. In this case, since $\sys(\partial K)$ is compact and $\partial K$ is Hausdorff, it follows that $\ev$ is a homeomorphism onto its image.

From the definition of the systolic $S^1$-index and that of generalized Zoll convex bodies, it follows that $K$ is generalized Zoll if and only if $\indfr(\sys(K)) \geq n$.

Assume that $K$ is Zoll. Then $\ev$ is a homeomorphism, which implies that $H^*(\sys(\partial K))$ is isomorphic to $H^*(S^{2n-1})$. From the $S^1$-fiber bundle
\[
\pi : \sys(\partial K) \times ES^1 \to \sys(\partial K) \times_{S^1} ES^1,
\]
we obtain the Gysin sequence
\[
\cdots \xrightarrow{\pi^*} H^{*+1}(\sys(\partial K)) \xrightarrow{\pi_*} H_{S^1}^*(\sys(\partial K)) \xrightarrow{\smile e} H_{S^1}^{*+2}(\sys(\partial K)) \xrightarrow{\pi^*} H^{*+2}(\sys(\partial K)) \xrightarrow{\pi_*} \cdots
\]
Since $H^*(\sys(\partial K)) \cong H^*(S^{2n-1})$, we conclude that $e^{n-1} \neq 0$, where $e$ denotes the fundamental class in $H_{S^1}^*(\sys(\partial K))$. Therefore,
\[
\indfr(\sys(\partial K)) \geq n.
\]

Now assume that $K$ is not Zoll. Denote by $P$ the image of $\sys(\partial K)$ under $\ev$. Since $\ev$ is a homeomorphism onto its image, it follows that 
\begin{equation}\label{isomorphismimageuniquness}
    H^*(\sys(\partial K)) \cong H^*(P).
\end{equation}

We will show that $H^{2i-1}(P) = \{0\}$ for all $i \geq n$, which implies that 
\[
\indsys(K) = \indfr(\sys(\partial K)) < n.
\] 
Indeed, from \eqref{isomorphismimageuniquness} we have $H^{2i-1}(\sys(\partial K)) = \{0\}$ for all $i \geq n$. Since $\indfr(\sys(\partial K))$ is finite, we can assume that $e^{r-1} \neq 0$ and $e^r = 0$, i.e., $\indfr(\sys(\partial K)) = r$. From the Gysin sequence 
\[
\cdots \xrightarrow{\pi^*} H^{2r-1}(\sys(\partial K)) \xrightarrow{\pi_*} H_{S^1}^{2r-2}(\sys(\partial K)) \xrightarrow{\smile e} H_{S^1}^{2r}(\sys(\partial K)) \xrightarrow{\pi^*} H^{2r}(\sys(\partial K)) \xrightarrow{\pi_*} \cdots
\]
and the assumptions $e^{r-1} \neq 0$ and $e^r = 0$, we see that
\[
\smile e : H_{S^1}^{2r-2}(\sys(\partial K)) \to H_{S^1}^{2r}(\sys(\partial K))
\]
has a nontrivial kernel, which coincides with the image of 
\[
\pi_* : H^{2r-1}(\sys(\partial K)) \to H_{S^1}^{2r}(\sys(\partial K)).
\]
Hence $H^{2r-1}(\sys(\partial K))$ must be nontrivial. Since $H^{2i-1}(\sys(\partial K)) = \{0\}$ for all $i \geq n$, we conclude that $e^{n-1} = 0$, and therefore $\indfr(\sys(\partial K)) < n$.

It remains to show that $H^{2i-1}(P) = 0$ for all $i \geq n$. Since $P$ is compact, we can cover it by open sets $\{U_j\}$, which are arbitrarily small neighborhoods of $P$ in $\partial K$ such that $U_j \neq \partial K$. Since $\partial K$ with one point removed is homeomorphic to $\mathbb{R}^{2n+1}$, each $U_j$ can be regarded as an open subset of $\mathbb{R}^{2n+1}$. It follows easily that $H^{2i-1}(U_j) = \{0\}$ for all $i \geq n$. 

By the continuity property of Alexandrov–Spanier cohomology, we have
\[
H^*(P) = \varinjlim\limits_{j} H^*(U_j).
\]
Hence, $H^{2i-1}(P) = 0$ for all $i \geq n$, which concludes the proof.

\textbf{Claim} (3): It follows from the fact that the space of generalized Zoll convex bodies in $\mathbb{R}^{2n}$ is the inverse image of the set $[n, +\infty)$ by $\indsys$, and $\indsys$ is upper semi-continuous by claim (1) of Proposition \ref{uppersemicontuperbound}.

\end{proof}

The fact that in the smooth case a convex body $C$ is generalized Zoll if and only if it is Zoll implies an upper bound for $\indsys$ in the smooth case. 

\begin{corollary}\label{systolics1indexnupperboundsmooth}
If $K \subset \mathbb{R}^{2n}$ is a convex body that satisfies the uniqueness of systoles property, then $\indsys(K) \leq n$, with equality if and only if $K$ is Zoll.
\end{corollary}

\begin{proof}
From Theorem~\ref{maingeneralizedzoll}, we know that for a convex body $K$ satisfying the uniqueness of systoles property, $\indsys(K) \geq n$ if and only if $K$ is Zoll. 
Therefore, it remains to show that for a Zoll convex body $K$, we have $\indsys(K) \leq n$. 

In this case, $\sys(\partial K)$ is homeomorphic to $S^{2n-1}$. 
Hence, $H^{2i-1}(\sys(\partial K)) = 0$ for $i > n$. 
From the $S^1$-fiber bundle
\[
\pi: \sys(\partial K) \times ES^1 \to \sys(\partial K) \times_{S^1} ES^1,
\]
we obtain the Gysin sequence
\[
\cdots \xrightarrow{\pi^*} H^{*+1}(\sys(\partial K)) 
\xrightarrow{\pi_*} H_{S^1}^*(\sys(\partial K)) 
\xrightarrow{\smile e} H_{S^1}^{*+2}(\sys(\partial K)) 
\xrightarrow{\pi^*} H^{*+2}(\sys(\partial K)) 
\xrightarrow{\pi_*} \cdots
\]
which implies that $e^{n} = 0$, and therefore $\indsys(K) = \indfr(\sys(\partial K)) \leq n$.
\end{proof}

\begin{corollary}\label{systolics1indexnupperboundS1}
Let $K \subset \mathbb{R}^{2n}$ be a convex body which is invariant under the standard $S^1$-action. Then $\indsys(K) \leq n$, where equality holds if and only if $K$ is a ball. In particular, the only $S^1$-invariant generalized Zoll convex body is a ball.      
\end{corollary}

\begin{proof}
From statement~(3) of Proposition~\ref{uppersemicontuperbound}, it follows that it is sufficient to show that the only $S^1$-invariant generalized Zoll convex body is a ball. 

Assume that $K \subset \mathbb{R}^{2n}$ is an $S^1$-invariant generalized Zoll convex body which is not a ball. Let
\[
\overline{B}(r) = \{ z \in \mathbb{K}^n \mid \pi \|z\| \leq r \}
\]
denote the ball tangent to the boundary of $K$. We know that
\[
c_1^{GH}(\overline{B}(r)) = c_1(K).
\]
For the proof of this fact, see \cite{GHR22}.

Since $K$ differs from the ball, we conclude that there exists a smooth convex body $\widetilde{K}$ with strictly larger volume than $\overline{B}(r)$ such that
\[
\overline{B}(r) \subset \widetilde{K} \subset K.
\]
Therefore,
\[
c_1^{GH}(\overline{B}(r)) = c_1(\widetilde{K}) = c_1(K) = c_n(K),
\]
which implies that $c_1(\widetilde{K}) = c_n(\widetilde{K})$, and by Theorem~\ref{maingeneralizedzoll}, $\widetilde{K}$ is Zoll.

Moreover, since $c_1(\widetilde{K}) = c_1(\overline{B}(r))$ and $\overline{B}(r)$ is a strict subset of $\widetilde{K}$, it follows that the systolic ratio of $\widetilde{K}$ is smaller than that of the ball, i.e.,
\[
\frac{c_1^{GH}(K)^n}{n!\operatorname{Vol}(K)} < \frac{c_1^{GH}(\overline{B}(r))^n}{n!\operatorname{Vol}(\overline{B}(r))} = 1.
\]
This is a contradiction, since $\widetilde{K}$ is Zoll and therefore has the same systolic ratio as a ball.

\end{proof}

\subsection{Evaluation map and generalized Zoll convex bodies}

We define the evaluation map  
\[
\ev : (\sys(\partial K), \|\cdot\|_{\infty}) \to (\partial K, |\cdot|), \quad \ev(\gamma) = \gamma(0).
\]
This map is continuous.

If $K \subset \mathbb{R}^{2n}$ is a convex body satisfying the uniqueness of systoles property and whose interior contains the origin, then $K$ is generalized Zoll if and only if $\ev(\sys(\partial K)) = \partial K$, or equivalently, if $\ev$ is a homeomorphism. 
This follows from statement~(2) of Theorem~\ref{maingeneralizedzoll}, which asserts that, in this case, $K$ is generalized Zoll if and only if it is Zoll. 
As we shall see, this equivalence no longer holds beyond this setting. 
In particular, the condition $\ev(\sys(\partial K)) = \partial K$ is not sufficient, as $\ev$ fail to be injective.

\begin{example}\label{polydicsevaluation}
The polydisc  
\[
P(a, \dots, a) = \left\{ z \in \mathbb{C}^n \mid \pi |z_i|^2 \leq a, \quad i \in \{1, \dots, n\} \right\},
\]
where $a > 0$, is not a generalized Zoll convex body, and yet 
\[
\ev(\sys(\partial P(a, \dots, a))) = \partial P(a, \dots, a).
\]

The fact that the polydisc is not generalized Zoll follows from the identity $c_i^{GH}(P(a, \dots, a)) = k a$. 
Moreover, $\ev(\sys(\partial P(a, \dots, a))) = \partial P(a, \dots, a)$ holds because through every point of $\partial P(a, \dots, a)$ there passes a systole. 
We now examine the case of the polydisc $P(1, 1)$ more explicitly. 
The space $\sys(\partial P(1, 1))$ consists of loops of the form
\[
\gamma: \mathbb{T} \to \partial P(1, 1), \quad \gamma(t) = \left( e^{2\pi i t} z_1, z_2 \right), \quad |z_1|^2 = \frac{1}{\pi}, \ |z_2|^2 \leq \frac{1}{\pi},
\]
and
\[
\gamma: \mathbb{T} \to \partial P(1, 1), \quad \gamma(t) = \left( z_1, e^{2\pi i t} z_2 \right), \quad |z_1|^2 \leq \frac{1}{\pi}, \ |z_2|^2 = \frac{1}{\pi}.
\]
It follows that $\ev(\sys(\partial P(1, 1))) = \partial P(1, 1)$. 
Furthermore, this space is $S^1$-homotopic to a space consisting of two disjoint orbits. 
Hence,
\[
\indsys(P(1,1)) = \indfr(\sys(\partial P(1,1))) = 1 < 2,
\]
and therefore $P(1,1)$ is not generalized Zoll.
\end{example}

The previous example shows that a convex body $C$ can have a systole passing through every point on $\partial C$ and still fail to be generalized Zoll. On the other hand, we now provide an example of a generalized Zoll convex body that does not satisfy the uniqueness of systoles property, yet still has the feature that through every point on its boundary there exists a generalized systole passing through that point.

\begin{example}\label{nonsmoothzollevaluation}
The convex body $B_\infty \times B_1 \subset \mathbb{R}^4$ is a generalized Zoll convex body that does not satisfy the uniqueness of systoles property, and 
\[
\ev(\sys(B_\infty \times B_1)) = \partial(B_\infty \times B_1).
\]

Since $B_\infty \times B_1$ is a Lagrangian product whose interior is symplectomorphic to a ball, Corollary~\ref{zollball} implies that it is generalized Zoll. On the other hand, from the proof of Proposition~\ref{w11norm}, it is clear that $B_\infty \times B_1$ does not satisfy the uniqueness of systoles property.

Let us now examine the dynamics on the boundary. The associated positively $2$-homogeneous Hamiltonian $H_{B_\infty \times B_1}: \mathbb{R}^{2n} \to \mathbb{R}$ is given by
\[
H_{B_\infty \times B_1}(x_1, x_2, y_1, y_2) = \max\left\{\|(x_1, x_2)\|_\infty^2, \|(y_1, y_2)\|_1^2\right\}.
\]

\begin{itemize}
    \item If $(x, y) \in \operatorname{int}(B_\infty) \times \partial B_1$, then 
    \[
    \partial H_{B_\infty \times B_1}(x, y) = \left(0, \partial \|(y_1, y_2)\|_1^2\right), \quad 
    J_0 \partial H_{B_\infty \times B_1}(x, y) = \left(-\partial \|(y_1, y_2)\|_1^2, 0\right).
    \]
    \item If $(x, y) \in \partial B_\infty \times \operatorname{int}(B_1)$, then 
    \[
    \partial H_{B_\infty \times B_1}(x, y) = \left(\partial \|(x_1, x_2)\|_\infty^2, 0\right), \quad 
    J_0 \partial H_{B_\infty \times B_1}(x, y) = \left(0, \partial \|(x_1, x_2)\|_\infty^2\right).
    \]
\end{itemize}

Therefore, if $(x_1, x_2, y_1, y_2)$ are such that $(x_1, x_2)$ is not on the diagonals of $B_\infty$ and $(y_1, y_2)$ is not on the diagonals of $B_1$, then the subdifferential contains a unique vector. This is not the case for points in $\partial B_\infty \times \partial B_1$ that lie off the diagonals, but in this situation, there is still only one vector in $J_0 \partial H_{B_\infty \times B_1}$ tangent to $\partial(B_\infty \times B_1)$. Hence, in such cases, the corresponding systole is unique, as illustrated in Figure~\ref{fig:linfl1nondiag}.

\begin{figure}[H]
  \centering
  \includegraphics[scale=0.3]{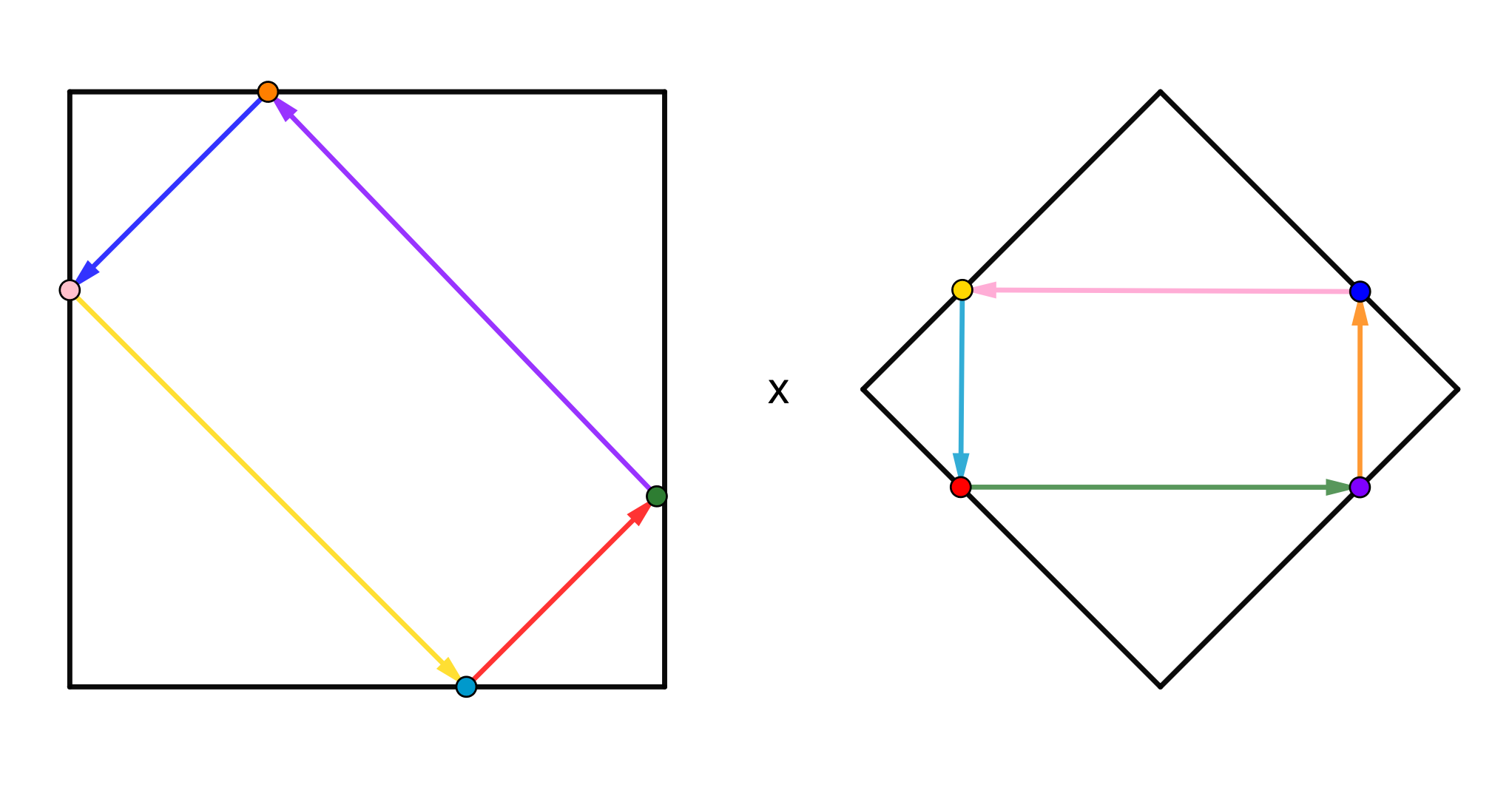}
  \caption{Systoles that do not contain the diagonals of $B_\infty$ or $B_1$.}
  \label{fig:linfl1nondiag}
\end{figure}

This figure illustrates the dynamics of the systoles. Since systoles must lie on the boundary of $B_\infty \times B_1$, it follows that either $(x_1, x_2)$ must lie on $\partial B_\infty$ or $(y_1, y_2)$ must lie on $\partial B_1$. We may assume that $(x_1, x_2)$ lies on $\partial B_\infty$ and is not one of its corners. 

For instance, let $(x_1, x_2)$ be the green node in Figure~\ref{fig:linfl1nondiag}. Then, whichever coordinate $(y_1, y_2)$ in $B_1$ we choose, it must move in the direction of the green vector until it reaches the boundary of $B_1$ (the purple node in the figure). At this point, the systole can no longer move in the direction of the green arrow, since it must remain on $\partial(B_\infty \times B_1)$. Therefore, $(y_1, y_2)$ remains constant. If we choose $(y_1, y_2)$ off the diagonals, there is only one direction in which $(x_1, x_2)$ can now move, namely, in the direction of the purple vector until it reaches the boundary of $B_\infty$ (the orange node in Figure~\ref{fig:linfl1nondiag}). This alternating process then continues until the loop closes.

By narrowing the rectangles in $B_1$ in the vertical direction, we obtain, in the uniform limit, the systoles illustrated in Figure~\ref{fig:linfl1onediag}.

\begin{figure}[H]
  \centering
  \includegraphics[scale=0.3]{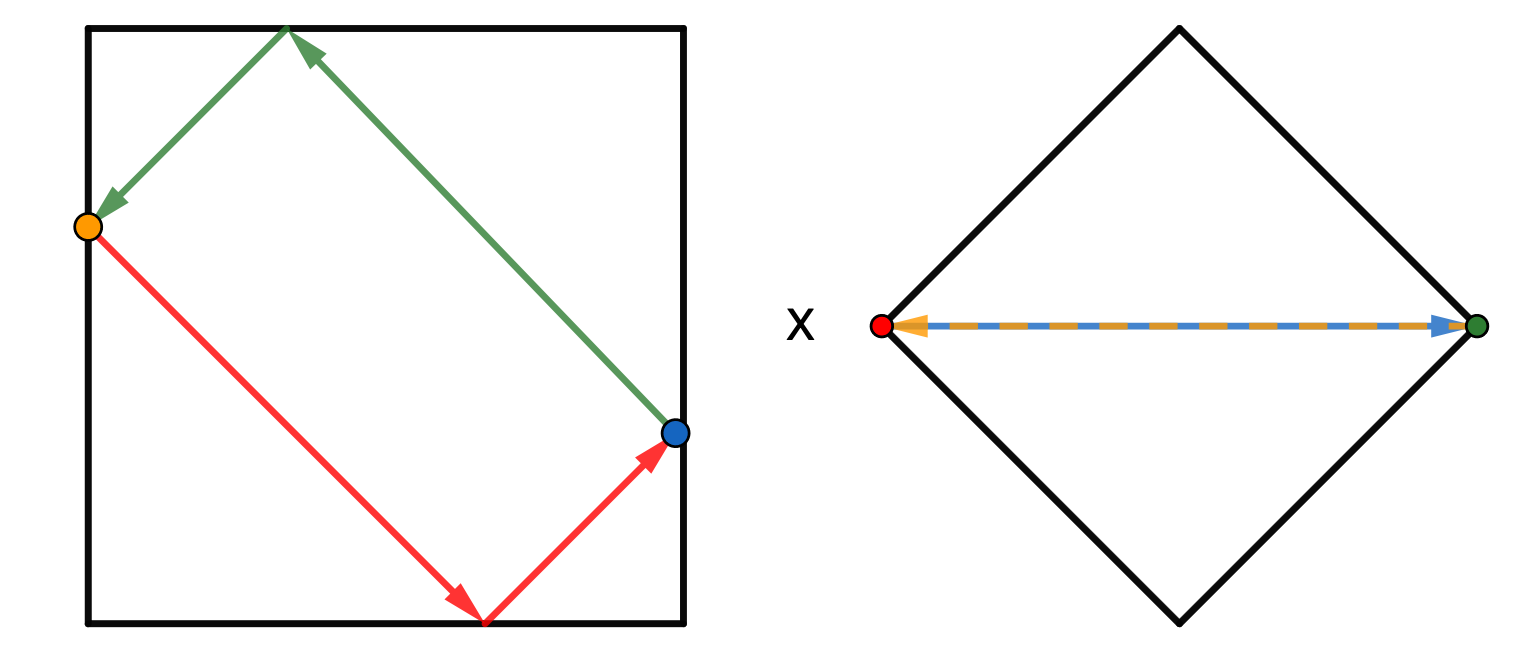}
  \caption{Systoles whose projection onto $B_1$ is the horizontal diagonal of $B_1$.}
  \label{fig:linfl1onediag}
\end{figure}

The dynamics of these systoles differ in the following way. Consider, for instance, the red node located at a corner of $\partial B_1$. If $(x_1, x_2) \in B_\infty$, the subdifferential of $H_{B_\infty \times B_1}$ at the point $(x_1, x_2, -1, 0)$ contains the convex hull of the vectors $(0, 0, -2, -2)$ and $(0, 0, -2, 2)$. As a result, the red arrow in $B_\infty$ can change direction at any point, taking any direction such that the angle measured from the horizontal axis lies within the range $[-\frac{\pi}{4}, \frac{\pi}{4}]$. Note that the directions of the red arrows in the figure correspond to these extremal directions.

By an analogous process, one can obtain systoles involving the other diagonals of $B_\infty$ and $B_1$.

By shrinking both rectangles, we obtain systoles that are entirely contained within the diagonals of $B_\infty$ and $B_1$. One such systole is illustrated in Figure~\ref{fig:linfl1twodiag}.

\begin{figure}[H]
  \centering
  \includegraphics[scale=0.3]{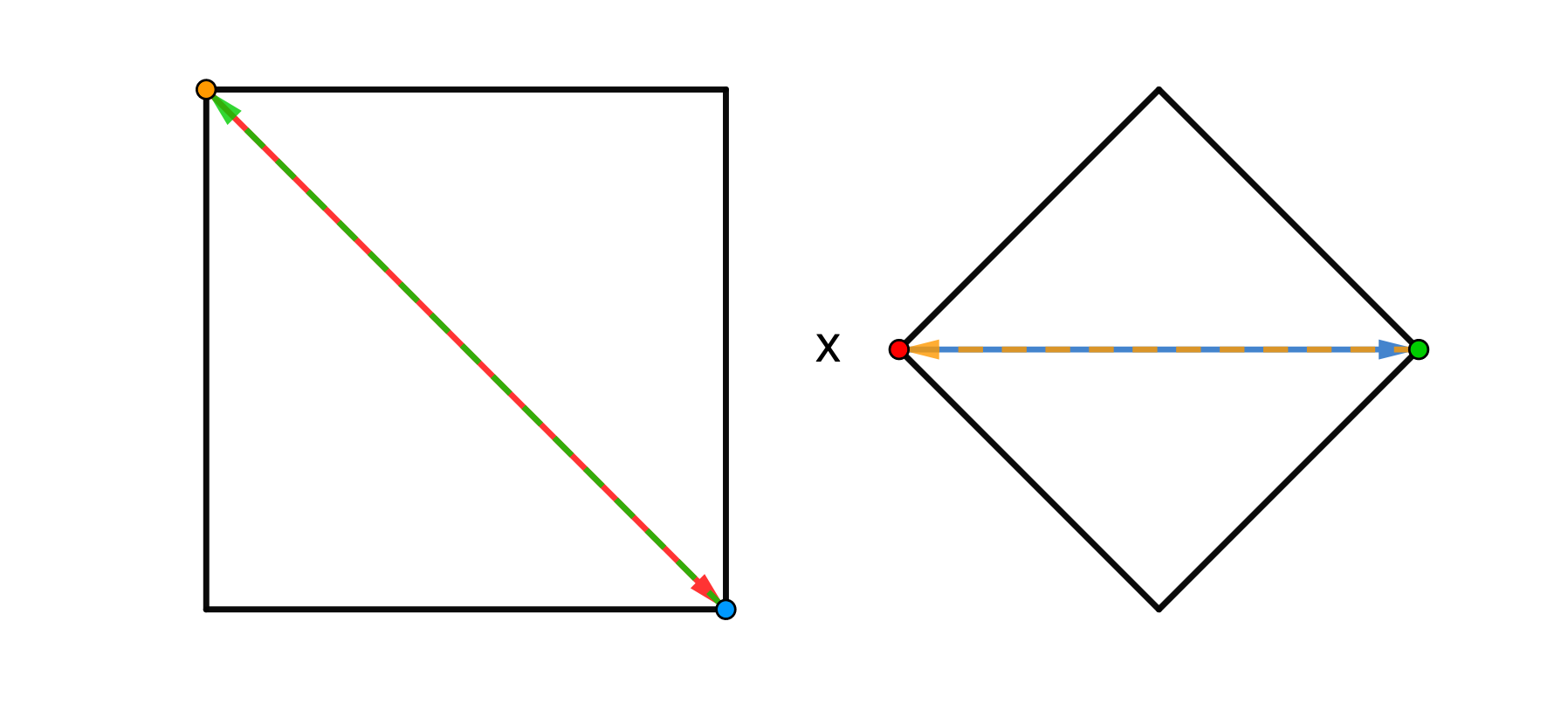}
  \caption{A diagonal systole on $\partial(B_\infty \times B_1)$.}
  \label{fig:linfl1twodiag}
\end{figure}

The subset of systoles 
\[
\widetilde{\sys}(\partial(B_\infty \times B_1)) \subseteq \sys(\partial(B_\infty \times B_1)),
\]
consisting of the systoles described above, satisfies 
\[
\ev(\widetilde{\sys}(\partial(B_\infty \times B_1))) = \partial(B_\infty \times B_1).
\]
Consequently,
\[
\ev(\sys(\partial(B_\infty \times B_1))) = \partial(B_\infty \times B_1).
\]

\end{example}

\end{document}